\documentclass[12pt]{amsart}
\usepackage{amssymb,xcolor, cancel, mathrsfs, comment, fullpage}

\usepackage[backref=page]{hyperref}

\newtheorem{theorem}{Theorem}[section]
\newtheorem{lemma}[theorem]{Lemma}
\newtheorem{proposition}[theorem]{Proposition}
\newtheorem{question}[theorem]{Question}
\newtheorem{conjecture}[theorem]{Conjecture}

\newtheorem{corollary}[theorem]{Corollary}

\theoremstyle{remark}
\newtheorem{remark}[theorem]{Remark}

\numberwithin{equation}{section}

\newcommand{\mo}{{-1}}

\newcommand{\bbZ}{\ensuremath{\mathbb{Z}}}

\newcommand{\calA}{\ensuremath{\mathcal{A}}}
\newcommand{\calE}{\ensuremath{\mathcal{E}}}
\newcommand{\calG}{\ensuremath{\mathcal{G}}}
\newcommand{\calI}{\ensuremath{\mathcal{I}}}

\newcommand{\calQ}{\ensuremath{\mathcal{Q}}}
\newcommand{\calY}{\ensuremath{\mathcal{Y}}}
\newcommand{\calS}{\ensuremath{\mathcal{S}}}

\newcommand{\frakA}{\ensuremath{\mathfrak{A}}}
\newcommand{\frakS}{\ensuremath{\mathfrak{S}}}
\newcommand{\bbF}{\ensuremath{\mathbb{F}}}
\newcommand{\bbFp}{\ensuremath{\mathbb{F}_p}}

\newcommand{\Aut}{\ensuremath{\mathrm{Aut}}}
\newcommand{\diam}{\ensuremath{\mathrm{diam}}}

\newcommand{\psl}{\ensuremath{\mathrm{PSL}}}

\newcommand{\sln}{\ensuremath{\mathrm{SL}}}
\newcommand{\gln}{\ensuremath{\mathrm{GL}}}

\newcommand{\bbC}{\ensuremath{\mathbb{C}}}

\begin{document}

\title{Expansion in Cayley graphs, Cayley sum graphs and their twists}

\author{Arindam Biswas}
\address{Department of Mathematical Sciences,
	University of Copenhagen,
	Universitetsparken 5,
	DK-2100 Copenhagen, Denmark}
\email{ab@math.ku.dk}
\thanks{}

\author{Jyoti Prakash Saha}
\address{Department of Mathematics, Indian Institute of Science Education and Research Bhopal, Bhopal Bypass Road, Bhauri, Bhopal 462066, Madhya Pradesh,
India}
\curraddr{}
\email{jpsaha@iiserb.ac.in}
\thanks{}



\maketitle

\begin{abstract}
The Cayley graphs of finite groups are known to provide several examples of families of expanders, and some of them are Ramanujan graphs. Babai studied isospectral non-isomorphic Cayley graphs of the dihedral groups. Lubotzky, Samuels and Vishne proved that there are isospectral non-isomorphic Cayley graphs of $\psl_d(\bbF_q)$ for every $d\geq 5$ ($d \neq 6$) and prime power $q> 2$. In this article, we focus on three variants of Cayley graphs, viz., the Cayley sum graphs, the twisted Cayley graphs, and the twisted Cayley sum graphs. We prove the existence of non-isomorphic expander families of bounded degree, whose spectra are related by the values of certain characters. We also provide several new examples of expander families, and examples of non-expanders and Ramanujan graphs formed by these three variants.  
\end{abstract}

\maketitle

\setcounter{tocdepth}{3} 


\section{Introduction}

Expanders are highly connected and sparse graphs, which have a number of applications in different branches of mathematics and computer science. They were first studied by Kolmogorov and Barzdin \cite{KolmogorovBarzdin}, and by Pinsker \cite{PinskerComplexityConcentrator}. The explicit construction of bounded degree expanders is an interesting and a challenging problem. There are now several constructions of examples of expanders using a wide variety of techniques. We refer to the survey articles by Hoory, Linial and Wigderson \cite{HooryLinialWigdersonBAMS}, Lubotzky \cite{LubotzkyExpanderGraphsBAMS}, and the monograph by Lubotzky \cite{LubotzkyDiscreteGroups}. 

\subsection{Expanders from Cayley graphs}
The first explicit construction of expanders is due to Margulis \cite{MargulisExpanders}. These were obtained from certain Cayley graphs of $\sln_{2}(\mathbb{Z}/p\mathbb{Z})$. Lubotzky--Phillips--Sarnak \cite{LPS88Ramanujan}, Margulis \cite{Margulis88ExpandConcent} and Chiu \cite{ChiuCubicRamanujan} proved that there exist arbitrarily large $(p+1)$-regular Ramanujan graphs for any prime $p$. Pizer used Hecke operators to obtain $(p+1)$-regular Ramanujan graphs \cite{PizerRamanujanGraphHeckeOperators}. These results were extended by Morgenstern to the $(q+1)$-regular Ramanujan graphs for every prime power $q$ \cite{MorgensternJCTB94}. Moreover, Lubotzky and Weiss have constructed examples of non-expanders using Cayley graphs \cite{LubotzkyWeissGroupsAndExpanders}, see also the work of Somlai \cite{SomlaiNonExpanderCayley}. The technique of constructing expanders by taking Cayley graphs of quotients of a suitable finitely generated group with respect to a finite generating set has been generalised by Shalom, where the connection set is not required to be a generating set \cite{ShalomExpandingGraphsInvariantMeans, ShalomExpandingGraphsAmenableQuotients}. Roichman proved that the Cayley graphs of symmetric groups (resp. alternating groups) with respect to the conjugacy class of long odd (resp. even) cycles are expanders, and the Cayley graphs of symmetric groups (resp. alternating groups) with respect to the conjugacy class of short odd (resp. even) cycles are not expanders \cite{RoichmanExpansionCayleyAlt}. Kassabov established that there are bounded-degree expanders formed by the Cayley graphs of symmetric groups, and by the Cayley graphs of alternating groups \cite{KassabovSymmetricGrpExpander}. In 1994, Alon and Roichman proved that random Cayley graphs are expanders \cite{AlonRoichmanRandomCayley}. Their result was further improved by Pak \cite{PakRandomCayleyGraphOlogG}, Landau--Russell \cite{LandauRussellRandomCayley}, Loh--Schulman \cite{LohSchulmanRandomCayley}, Christofides--Markstr\"{o}m \cite{ChristofidesMarkstromExpanRandomCayley}. The expansion properties of subsets of special linear groups, and of random pair of elements in finite simple groups have recently been studied by several authors. We refer to the works of Margulis \cite{MargulisExpanders}, Kassabov--Lubotzky--Nikolov \cite{KassabovLubotzkyNikolovFiniteSimpleGrpExp}, Helfgott \cite{HelfgottGrowthGenerationSL2Fp}, Bourgain--Gamburd \cite{BourgainGamburdUniformExpanBdd, BourgainGamburdExpansionRandomWalkSLdI, BourgainGamburdExpansionRandomWalkSLdII}, Liebeck--Shalev \cite{LiebeckShalevDiamFiniSimpleGrp}, Bourgain--Gamburd--Sarnak \cite{BourgainGamburdSarnakAffineLinearSieveExpanderSumProd}, Breuillard--Gamburd \cite{BreuillardGamburdStrongUnifExpans}, Breuillard--Green--Tao \cite{BreuillardGreenTaoApproxSubgrpLinGrp}, Breuillard--Green--Tao \cite{BGTSuzuki}, Varj\'{u} \cite{VarjuExpansionSLdModISqFree}, Bourgain--Varj\'{u} \cite{BourgainVarjuExpansionSLdZqZ}, Golsefidy--Varj\'{u} \cite{GolsefidyVarjuExpansionPerfectGrp}, Kowalski \cite{KowalskiExplicitGrowthExpansionSL2}, Helfgott--Seress  \cite{HelfgottSeressDiamPermGrp}, Breuillard--Green--Guralnick--Tao \cite{BGGTExpansionSimpleLie}, Pyber--Szab\'{o} \cite{PyberSzaboGrowthFinSimpleGrpLie}, Bradford \cite{BradfordExpansionRandomWalkSieving}. Furthermore, there are examples of expanders that have been constructed using other methods. We refer to the works of Gabber--Galil \cite{GabberGalilLinearSized},  Reingold--Vadhan--Wigderson \cite{ReingoldVadhanWigdersonEntropyZigZag}, Alon--Schwartz--Shapira \cite{AlonSchwartzShapiraElemConsContDegreeExpander}, Marcus--Spielman--Srivastava \cite{MarcusSpielmanSrivastavaInterlaFam1}, Alon \cite{AlonExplicitExpandersEveryDegSize}. Moreover, there are various notions of expansion in the literature. We refer to the works of Dvir--Wigderson \cite{DvirWigdersonMonotoneExpanders} and Bourgain--Yehudayoff \cite{BourgainYehudayoffExpansionInSL2MonotoneExpa} on monotone expanders, Lubotzky--Zelmanov \cite{LubotzkyZelmanovDimensionExpanders} and Dvir--Shpilka \cite{DvirShpilkaTowardsDimExpanderFiniteField} on dimension expanders, and the survey of Lubotzky on \cite{LubotzkyICMHighDimensionalExpander} high dimensional expanders. 

\subsection{Variants of Cayley graphs}
Let $G$ be a finite group and $S$ be a subset of $G$. Then the Cayley graph $C(G, S)$ (resp. the Cayley sum graph $C_\Sigma(G, S)$) is the graph having $G$ as its set of vertices, and for $x, y\in G$, there is an edge from $x$ to $y$ if $y = xs$ (resp. $y = x^\mo s$). Cayley graphs have been extensively studied over the ages. However, despite being classical combinatorial objects, the literature on Cayley sum graphs is few and most of the works are quite recent. Indeed, even the question of vertex connectivity of abelian Cayley sum graphs was treated very recently in 2009 by Grynkiewicz--Lev--Serra \cite{GrynkiewiczLevSerraConnCaylSum}. One reason for this is that Cayley sum graphs may have less symmetry in them unlike Cayley graphs. For instance, Cayley graphs are vertex transitive whereas Cayley sum graphs need not be so. Even less is known about the spectra of Cayley sum graphs, whereas the computation of distribution of eigenvalues of graphs is a fundamental topic of interest in graph theory. It is clear that much remains to be discovered about Cayley sum graphs. 

Given an automorphism $\sigma$ of a finite group $G$ and a subset $S$ of $G$, one can consider the variants of the Cayley graph $C(G, S)$ and the Cayley sum graph $C_\Sigma(G, S)$, as introduced in \cite{CheegerTwisted}. The twisted Cayley graph $C(G, S)^\sigma$ (resp. the twisted Cayley sum graph $C_\Sigma(G, S)^\sigma$) is a graph having $G$ as its set of vertices, and for $x, y\in G$, there is an edge from $x$ to $y$ if $y = \sigma(xs)$ (resp. $y = \sigma(x^\mo s)$). Henceforth, we assume that $|G| \geq 4$.

The three variants of Cayley graphs considered above are different from Cayley graphs. Let $p, q$ be odd primes with $p > 3q$. The Cayley sum graph of the permutation group $\frakS_p$ with respect to the set $S$ of cycles of length $q$ admits $3$-cycles at certain vertices (for instance, at the vertices corresponding to the cycles of length $q$) and does not admit $3$-cycles at several other vertices (for instance, at the vertices corresponding to the $p$-cycles). Moreover, if $\sigma$ denotes the inner automorphism of $\frakS_p$ corresponding to the transposition $(1, 2)$, then the twisted Cayley sum graph $C_\Sigma(\frakS_p, S)^\sigma$ admits $3$-cycles at certain vertices (for instance, at the vertices corresponding to the cycles of length $q$ that fixes $1$ and $2$) and does not admit $3$-cycles at several other vertices (for instance, at the vertices corresponding to the $p$-cycles which sends $1$ to $2$ or $2$ to $1$). Also note that the number of $3$-cycles at a vertex $v$ of the twisted Cayley graph $C(\frakS_p, S)^\tau$ depends on $v$ where $\tau$ denotes the inner automorphism of $\frakS_p$ induced by the element $(1, 2)(3, 4) \cdots (p-2,p-1)$ of $\frakS_p$, for instance, there are no $3$-cycles at $(p, 1, 3, 5, \ldots, p-2)$, and there are $3$-cycles at $(p, 1, 3, 5, \ldots, \frac{q-1}2)$. This shows that the three variants of Cayley graphs are different from Cayley graphs. 

\subsection{Isospectral graphs from Cayley graphs}
Two finite graphs having the same number of vertices are called \textit{isospectral} if the multiset of eigenvalues of their adjacency operators are the same. Babai proved that for any prime $p$, the dihedral group of order $2p$ admits $p/64$ pairs of generators such that the associated Cayley graphs are isospectral and non-isomorphic \cite{BabaiSpectraofCayley}. Lubotzky, Samuels and Vishne proved that using this result, examples of dense isospectral non-isomorphic Cayley graphs can be obtained \cite[Corollary 5.2]{LubotzkySamuelsVishneIsospectralCayleyGraphs}. Moreover, they proved that there are families of groups $G_n$ having two sets of generators $A_n, B_n$ of bounded size such that the Cayley graphs $C(G_n, A_n), C(G_n, B_n)$ are isospectral and non-isomorphic \cite[Theorem 1]{LubotzkySamuelsVishneIsospectralCayleyGraphs}. They also showed that for certain integers $r$, and for any $k\geq 1$, there exist $k$ isospectral non-isomorphic $r$-regular graphs \cite[Theorem 2]{LubotzkySamuelsVishneIsospectralCayleyGraphs}. 

It would be interesting to consider refinements of the problem of finding examples of non-isomorphic isospectral graphs. For instance, one could look for isospectral non-isomorphic Ramanujan graphs. One could also look for two (or more) families of isospectral non-isomorphic graphs, possibly with a weaker condition on the expansion. 

\begin{question}
\label{Qn:TwoIsospecNonIsom}
Do there exist two families $\{(V_n, E_n)\}_{n\geq 1}, \{(V_n, E_n')\}_{n\geq 1}$ of expanders of bounded degree such that for any $n\geq 1$, the graphs $(V_n, E_n), (V_n, E_n')$ are isospectral but non-isomorphic? 
\end{question}

The notion of two (or more) operators being isospectral can be refined further, as described below. 

If $V$ is a vector space carrying an action of an abelian group $H$ and $T:V \to V$ is an operator commuting with the action of every element of $H$, then for any character $\chi$ of $H$, the spectrum of the restriction of $T$ to the $\chi$-eigenspace $V_\chi$ of $H$ is denoted by $\calE(T)_\chi$. If $V$ is a vector space carrying an action of an abelian group $H$, then a collection $\{T_h\}_{h\in H}$ of operators on a vector space $V$, indexed by the elements of $H$, is said to be \textit{$H$-isospectral} if $T_h$ commutes with the action of any element of $H$ on $V$ for any $h\in H$, and for any character $\chi$ of $H$, $\overline \chi(h) \calE(T_h)_\chi = \overline \chi(h') \calE(T_{h'})_\chi $ holds for any $h, h'\in H$ (or equivalently, $\calE(T_h)_\chi = \chi(h) \calE(T_{e})_\chi $ for any $h\in H$). Note that $k$ operators $T_0, T_1, \ldots, T_{k-1}$ on a vector space $V$ are isospectral if and only if they are $\bbZ/k\bbZ$-isospectral where $\bbZ/k\bbZ$ acts on $V$ trivially. 

Examples of $H$-isospectral collections of operators can be obtained from twisted Cayley graphs. For instance, if $H$ is a $2$-torsion group acting on a group $G$ through group automorphisms, then the collection of the adjacency operators of the twisted Cayley graphs $\{C(G, S)^h\}_{h\in H}$ is $H$-isospectral if $S$ is symmetric and closed under the action of $H$ (Proposition \ref{Prop:IsospecExample}). It will be convenient to have the notion of $H$-isospectrality extended to a collection of graphs indexed by the elements of $H$. A collection $\{\Gamma_h\}_{h\in H}$ of graphs on the same vertex set $V$ is said to be \textit{$H$-isospectral} if $H$ acts on $V$, and the collection formed by the adjacency operators $A_h$ of $\Gamma_h$ is $H$-isospectral. Thus, by Proposition \ref{Prop:IsospecExample}, the collection $\{C(G, S)^h\}_{h\in H}$ of twisted Cayley graphs is $H$-isospectral under suitable conditions. Note that the result of Lubotzky, Samuels and Vishne on the existence of $k$ isospectral non-isomorphic graphs is equivalent to stating that there exists $\bbZ/k\bbZ$-isospectral non-isomorphic graphs where $\bbZ/k\bbZ$ acts trivially on the vertex sets. 

Instead of looking for $\bbZ/2\bbZ$-isospectral non-isomorphic expander families (as asked in Question \ref{Qn:TwoIsospecNonIsom}), we may try to have $H$-isospectral expander families, i.e., an expander family $\{\Gamma_{hn}\}_{n\geq 1}$ for every $h\in H$ such that for any $n$, the collection of graphs $\{\Gamma_{hn}\}_{h\in H}$ have the same vertex set $V_n$, and is $H$-isospectral. Such an $H$-isospectral expander families is said to be \textit{non-isomorphic} (resp. $s$-\textit{non-isomorphic}) if for any $n\geq 1$, the graphs $\Gamma_{hn}, \Gamma_{h'n}$ are non-isomorphic for any two distinct elements $h, h'$ of $H$ (resp. for any two distinct elements $h, h'$ lying in a certain size $s$ subset of $H$, which is independent of $n$). 

If each term of a sequence $\{V_n\}_{n\geq 1}$ of vector spaces carries an action of $H$ through linear automorphisms and $\lim\dim V_n = \infty$, then $H$ is said to act \textit{uniformly} on this sequence if the dimension of the $\chi$-isotypic component of $V_n$ does not depend on the character $\chi$ of $H$ as $n \to \infty$, i.e., 
$$
\lim_{n\to \infty} 
\frac 
{\dim V_{n, \chi}}
{\dim V_n}
= \frac 1{|H|}
$$
holds for any character $\chi$ of $H$. For example, the action of $\bbZ/2\bbZ \times \bbZ/2\bbZ$ on $\frakS_n$ (where $(1, 0)$ (resp. $(0,1)$) acts through conjugation by $(1, 2)$ (resp. $(3, 4)$)) is uniform (Lemma \ref{Lemma:UniformOnSn}). 
A sequence of expander families $\{\{\Gamma_{hn}\}_{n\geq 1}\}_{h\in H}$ indexed by $H$, is said to be $H$-\textit{uniform} if the $n$-th layer of graphs $\{\Gamma_{hn}\}_{h\in H}$ have the same vertex set $V_n$ carrying an action of $H$ and $H$ acts uniformly on $\{L^2(V_n)\}_{n\geq 1}$.

The two conditions for a collection of expander families being $H$-isospectral and uniform have been brought into play at several occasions in order to obtain results on the relation among the spectra of adjacency operators restricted to certain isotypic components and results on the dimensions of these isotypic components. 

\subsection{Results obtained}
In this article, we study the three variants of Cayley graphs. We use them to construct examples of graphs and families of graphs with interesting properties, and we also obtain results about certain properties of these variants.  

For any $2$-torsion group $H$, we prove the existence of uniformly $H$-isospectral $s$-non-isomorphic expander families of bounded degree with $s = \log |H|$ (\S \ref{Sec:AsympIsoSpect}). Further, one can provide several new examples of expanders (\S \ref{Sec:Expanders}), and examples of non-expanders (\S \ref{Sec:NonExpander}) and Ramanujan graphs (\S \ref{Sec:Ramanujan}) formed by these variants. Most of these examples have been constructed by considering certain appropriate twisted Cayley graphs. Some of the examples of Ramanujan graphs that have been obtained, are non-isomorphic to the Ramanujan graphs constructed by Lubotzky, Phillips and Sarnak \cite{LPS88Ramanujan}, unless the spectra of the graphs on $\psl_2(\bbF_q)$ constructed by them contain large subsets, symmetric about the origin. We also study the diameters of the three variants of Cayley graphs (\S \ref{Sec:Diam}). 

A crucial ingredient of our study is the relation between the graphs $C(G, S)$, $C_\Sigma(G, S)$, $C(G, S)^\sigma$, $C_\Sigma(G, S)^\sigma$. Note that their square graphs are equal under certain hypotheses. This shows that the eigenvalues of any two of these graphs are related by factors of $\pm 1$. However, this does not allow us to have any information about the total number of factors of $1$ (and $-1$) required to obtain the eigenvalues of one of them  from the eigenvalues of another one. The problem remains to pin down the factors of $1$ required to determine the eigenvalues of the variants of Cayley graphs from that of the Cayley graph. While dealing with Cayley graphs, the representation-theoretic techniques become useful since the adjacency operator of a Cayley graph $C(G, S)$ of a group $G$ is related to the action of certain linear operator (viz., the right-multiplication by $\sum_{s\in S} s$) on the group algebra $\bbC[G]$. A major difficulty in dealing with the three variants of Cayley graphs is that there does not seem to be any way to realise their adjacency operators through certain linear operators on the group algebra. We circumvent this problem by relating the adjacency operators of Cayley graphs, Cayley sum graphs, twisted Cayley graph and twisted Cayley sum graphs through certain involutions (under some hypothesis). This relation between the adjacency operators of two such graphs yields a relation between their spectra indicating the number of factors of $\pm 1$ that relates them (Theorem \ref{Thm:EquivExpansion}). This proves a sought after result for Cayley sum graphs \cite[p. 112]{LubotzkyDiscreteGroups} (Theorem \ref{Thm:SpectrumThruChar}). One also obtains further results, as described below.

\subsubsection{Uniformly $H$-isospectral $s$-non-isomorphic expander families}

The construction of Lubotzky, Samuels and Vishne \cite{LubotzkySamuelsVishneIsospectralCayleyGraphs} of isospectral non-isomorphic Cayley graphs motivates the question of finding $H$-isospectral $s$-non-isomorphic expander families. We explain how the above-mentioned technique can be applied to yield such graphs. Consider the Cayley graph $C(G, S)$ of a group $G$ with respect to a subset $S$ such that $C(G, S)$ is undirected (which amounts to replacing $S$ by the set $S \cup S^\mo$ having size $\leq 2|S|$). If $\sigma$ is an order two automorphism of $G$ having very few fixed points, then by Theorem \ref{Thm:EquivExpansion}, around $50\%$ eigenvalues of the twisted Cayley graph $C(G, S)^\sigma$ are equal to around $50\%$ eigenvalues of $C(G, S)$, and the remaining eigenvalues of $C(G, S)^\sigma$ are equal to the negatives of the remaining eigenvalues of $C(G, S)$. It turns out that the collection $\{C(G, S), C(G, S)^\sigma\}$ is a $\bbZ/2\bbZ$-isospectral collection of graphs where the non-trivial element of $\bbZ/2\bbZ$ acts on $G$ by $\sigma$ (assuming $C(G, S)^\sigma$ to be undirected, which amounts to replacing $S$ by the set $S \cup \sigma(S^\mo)$ of size $\leq 2|S|$). More generally, given a $2$-torsion group $H$, consisting of automorphisms of $G$ of order $\leq 2$, the set $S$ can be enlarged to a set $T$ of size $\leq 2^k|S|$ such that the collection $\{C(G, S)^{h}\}_{h\in H}$ of twisted Cayley graphs is $H$-isospectral. This can be summarised as follows. Given an undirected Cayley graph $C(G, S)$ and an integer $k\geq 1$, one can obtain $k$ graphs (provided $\Aut(G)$ contains a $2$-torsion subgroup of order $k$), by replacing $S$ by a set $T$ having size $\leq 2^k|S|$, such that each of these graphs works as a ``square root'' of the square graph of $C(G, S)$, and hence having eigenvalues equal to the eigenvalues of $C(G, S)$, up to factors of $\pm 1$. Relying on available constructions of expanders, this strategy yields that for any $k\geq 2$, there exists uniformly $H$-isospectral expander families of bounded degree formed by the twisted Cayley graphs of symmetric groups, of alternating groups, of special linear groups, and of projective special linear groups (Corollaries  \ref{Cor:AsymptoticallyIsoSpecPerm}, \ref{Cor:AsymptoticallyIsoSpecSLnPSLd}). 

The next step is to look at those twists which are non-isomorphic. It turns out that the twisted Cayley graphs of $G$ with respect to the inner automorphisms corresponding to two order-two elements of $G$ are isomorphic if those elements of $G$ are conjugate to each other (see Proposition \ref{Prop:IsomTwistedCayley} for the precise statement). Therefore, to obtain non-isomorphic $H$-isospectral expanders, we consider the twists with respect to inner automorphisms of permutation groups corresponding to involutions of different cycle types, and prove that such considerations yield non-isomorphic graphs. This strategy combined with a result of Kassabov \cite{KassabovSymmetricGrpExpander} yields that for any $2$-torsion group $H$, there exist uniformly $H$-isospectral $\log |H|$-non-isomorphic expanders of bounded degree formed by the twisted Cayley graphs of symmetric groups, and by the twisted Cayley graphs of alternating groups (Theorem \ref{Thm:AsymptoticallyIsoSpecNonIsom}). 

\subsubsection{Construction of expanders}
We provide certain explicit construction of expanders, that remained unknown in the literature, except for the construction of expanders using Cayley sum graphs of cyclic groups by Chung \cite{Chung89JAMS}. This allows us to shed some light on the problem of finding new methods to construct regular expanders, proposed by Lubotzky \cite[Problem 10.1.3]{LubotzkyDiscreteGroups}. This construction has a distinctive feature as compared to the known examples of expanders afforded by Cayley graphs. The expanders constructed using Cayley graphs are vertex transitive. However, except for the expanders afforded by the cyclic groups as constructed by Chung, there does not seem to be any known example of non-vertex-transitive regular expanders of group-theoretic origin, for instance, non-vertex-transitive expanders afforded by permutations groups, or by finite simple groups of Lie type. The constructions in this work provide plenty of examples of non-vertex-transitive regular expanders, and rely on available constructions of expanders using Cayley graphs and Schreier graphs. We refer to \S \ref{Sec:Expanders} for the precise statements. 

Further, we show that the technique of twisting can be used to obtain explicit examples of expanders from the expanders constructed by Gabber and Galil. This leads to examples of uniform $(\bbZ/2\bbZ)^2$-isospectral expander families of degree $8$ (\S \ref{SubSec:GabberGalil}). 

\subsubsection{Construction of Ramanujan graphs}

In \S \ref{Sec:Ramanujan}, we obtain examples of Ramanujan graphs. Note that there are very few constructions of Ramanujan groups, viz., the ones constructed by Lubotzky, Phillips and Sarnak \cite{LPS88Ramanujan}, Margulis \cite{Margulis88ExpandConcent}, Pizer \cite{PizerRamanujanGraphHeckeOperators}, Chiu \cite{ChiuCubicRamanujan}, Morgenstern \cite{MorgensternJCTB94}, and very recently, by Marcus--Spielman--Srivastava \cite{MarcusSpielmanSrivastavaInterlaFam1} in the bipartite case. We prove that if a certain portion of the spectra of the Ramanujan graphs on $\psl_2(\bbF_q)$ and on $\mathrm{PGL}_2(\bbF_q)$, constructed by Lubotzky, Phillips and Sarnak, is not symmetric about the origin, then our constructions yield new examples of Ramanujan graphs. More precisely, for a given prime $p\equiv 1 \pmod 4$, if the spectra of the Cayley graphs $C(G_q, S^{p,q})$ do not contain a `large' subset (having size around $75\%$ of the size of $G_q$), which is symmetric about the origin, then twisted Cayley graph $C(G_q, S^{p,q})^{\sigma_q}$ is non-isomorphic to $C(G_q, S^{p,q})$ for some suitable order two automorphism $\sigma_q$ of $G_q$ (we refer to Theorem \ref{Thm:75Percent} for the precise statement). 

\subsubsection{Diameter bounds}

In \S \ref{Sec:Diam}, using Theorem \ref{Thm:EquivExpansion}, we obtain bounds on the diameters of the three variants of Cayley graphs relying on results on diameter bounds of Cayley graphs.

\subsubsection{Distinguishing twisted Cayley graphs and a dichotomy result}

Most of the constructions of examples of expanders and Ramanujan graphs in \S  \ref{Sec:Ramanujan}, \ref{Sec:NonExpander}, \ref{Sec:Expanders}, \ref{Sec:AsympIsoSpect}, are the twisted Cayley graphs of certain groups with respect to certain order two automorphisms. Thus, it is important to be able to distinguish a twisted Cayley graph $C(G, S)^\sigma$ from the Cayley graph $C(G, S)$. In \S \ref{Sec:TwistedCayley}, we explain several strategies to distinguish a twisted  Cayley graph (or several such graphs) from the corresponding Cayley graph. In Proposition \ref{Prop:IsomTwistedCayley}, we provide a sufficient condition for showing that two twisted Cayley graphs are isomorphic. Next, we establish a dichotomy result which states that the spectrum of a Cayley graph is `too symmetric', or one among its twists by a `random enough' involution yields a non-isomorphic graph (Theorem \ref{Thm:75Symm}). More precisely, for any $k\geq 1$, the spectrum of the Cayley graph $C(G, S)$ contains a symmetric subset of size $\left( 1 - \frac 1 {2^k} - o(1)\right)|G|$, or for any subgroup $H$ of the automorphism group $\Aut(G)$ such that $H$ is isomorphic to $(\bbZ/2\bbZ)^k$ and the number of elements in $G$ having distinct images under the action of the elements of $H$ is equal to $(1 - o(1))|G|$, the graph $C(G, S)$ is non-isomorphic to the twisted Cayley graph $C(G, S)^\sigma$ for some $\sigma\in H$. 

\section{Expansion in one implies expansion in the other}

From now on, we assume that the automorphism $\sigma$ of $G$ is of order two, unless otherwise stated. It is known that each of the graphs $C(G, S)$, $C_\Sigma(G,S)$, $C(G, S)^\sigma$, $C_\Sigma(G, S)^\sigma$ has the property that ``combinatorial expansion implies spectral expansion''. For Cayley graphs, this property was first established qualitatively by Breuillard, Green, Guralnick and Tao \cite[Appendix E]{BGGTExpansionSimpleLie}. Later, this property was established quantitatively for the Cayley graphs and its three variants in \cite[Theorem 1.4]{BiswasCheegerCayley}, \cite[Theorem 1.3]{CheegerCayleySum}, \cite[Theorem 1.1]{CheegerTwisted}. An improved bound for Cayley graphs was obtained by Moorman, Ralli and Tetali \cite[Theorem 2.6]{CayleyBottomBipartite}. In this section, we prove that, under certain conditions, combinatorial expansion in any one graph $X$ among $C(G, S)$, $C_\Sigma(G,S)$, $C(G, S)^\sigma$, $C_\Sigma(G, S)^\sigma$ implies spectral expansion in any other graph $Y$ among them. We refer to Theorem \ref{Thm:IntroEquiv} for the precise statement. 

\begin{theorem}
\label{Thm:IntroEquiv}
Suppose $X, Y$ are two graphs among $C(G, S)$, $C_\Sigma(G,S)$, $C(G, S)^\sigma$, $C_\Sigma(G, S)^\sigma$ and assume that they are undirected and connected. In addition, assume that $S$ is symmetric if $X, Y$ are equal to $C_\Sigma(G, S), C(G, S)^\sigma$ in some order. If one of them is a vertex expander or a two-sided spectral expander, then the other is a vertex expander and a two-sided spectral expander. 
\end{theorem}

The above result follows from Theorem \ref{Thm:EquivExpansion}. We explain the underlying idea that is used to relate the spectra of these four graphs. Suppose $(V, E)$ is a graph and $\sigma:V \to V$ is a bijection. Consider the graph $(V, \calE)$, which has $V$ as its set of vertices, and the neighborhood of an element $v$ is defined to be the neighborhood of $\sigma (v)$ in $(V, E)$. If $A,\calA$ are the adjacency operators of $(V, E), (V, \calE)$, then it follows that $\calA= P A$ where $P$ is the permutation matrix $\sum_{v\in V} \delta_{v, \sigma (v)}$. Moreover, if $(V, E), (V, \calE)$ are undirected, then it follows that $PA = AP^t$. If $\sigma$ is of order two, then $P, A$ commute and hence up to factors of $\pm 1$, the eigenvalues of the adjacency operator $A$ of $(V, E)$ are equal to the eigenvalues of the adjacency operator $\calA$ of $(V, \calE)$. Moreover, the number of factors of $1$ is equal to the average of $|V|$ and the number of fixed points of $\sigma$.

\begin{lemma}
\label{Lemma:EigenvalueXY}
Let $X, Y$ be finite, undirected graphs having the same vertex set $V$. Denote the adjacency operator of $X$ (resp. $Y$) by $A_X$ (resp. $A_Y$). Suppose the adjacency operators satisfy $A_Y = P A_X$ where $P$ is an operator $P$ acting on $L^2(V)$ and $P$ induces a permutation of order $\leq 2$ on the set of functions $\delta_v$ with $v\in V$. Then, up to factors of $\pm 1$, the eigenvalues of the adjacency operator of $X$ are equal to the eigenvalues of the adjacency operator of $Y$. Moreover, if $P$ fixes $\delta_u$ for some $u\in V$, then the following statements hold. 
\begin{enumerate}
\item 
One of $X, Y$ is connected if and only if the other one is connected.
\item 
If the graphs $X, Y$ are connected, then their diameters satisfy
$$\diam(X) \leq 2 \diam(Y), \quad 
\diam(Y) \leq 2 \diam(X).$$
\item 
One of $X, Y$ is connected and bipartite if and only if the other one is connected and bipartite. 
\end{enumerate}
\end{lemma}

\begin{proof}
Since $X,Y$ are undirected, it follows that $A_X = A_X^t, A_Y = A_Y^t$, which yields $PA_X = A_X P$. Since the eigenvalues of $P$ belong to $\{1, -1\}$, it follows that up to factors of $\pm 1$, the eigenvalues of $A_X$ are equal to the eigenvalues of $A_Y$ in some order. 

Suppose $Y$ is connected. Let $v$ be a vertex of $Y$ such that $u$ is connected to $v$ through an odd number of edges. So the $(u,v)$-th entry of $A_Y^k$ is positive for some odd integer $k$. Since $P$ fixes $\delta_u$, it follows that the $(u,v)$-th entry of $PA_Y^k$ is positive. Since $P$ and $A_X$ commute and $P^2$ is the identity map, it follows that the $(u,v)$-th entry of $A_X^k$ is positive, which shows that there is a path joining $u, v$ in $X$. Note that the even powers of $A_X$ coincide with the even powers of $A_Y$. Hence, for $v'\in V$, there is a path of even length in $X$ joining $u, v'$ if and only if there is a path of even length in $Y$ joining $u, v'$. It follows that $X$ is connected. 

The bounds on the diameters of $X, Y$ follow from the above argument. 

Suppose $Y$ is connected and bipartite. Then $X$ is connected. If $X$ is not bipartite, then it admits an odd cycle at a vertex $v$. Since $X$ is connected, the vertices $u, v$ are connected in $X$. So, there is an odd cycle at $u$ in $X$, i.e., the $(u, u)$-th entry of $A_X^k$ is positive for some odd positive integer $k$. Since $P$ fixes $\delta_u$, it follows that the $(u, u)$-th entry of $A_Y^k = PA_X^k$ is also positive, which contradicts the bipartiteness of $Y$. Hence, $X$ is bipartite. 
\end{proof}

In the following, $\calS$ denotes a finite multi-subset of $G$. 

\begin{lemma}
\label{Lemma:CayleyTwistedUndir}
The graph $C(G, \calS)^\sigma$ is undirected if and only if $\calS$ is equal to $\sigma(\calS^\mo)$. Moreover, the twisted Cayley sum graph $C_\Sigma (G, \calS)^\sigma$ is undirected if and only if $\calS$ is equal to $g\sigma(\calS)g^\mo$ for any $g\in G$. 
\end{lemma}

\begin{proof}
For $x\in G, s\in \calS$ and $y\in G$ with $y = \sigma(xs)$, we have $x = \sigma(y\sigma(s)^\mo)$. Hence the graph $C(G, S)^\sigma$ is undirected if and only if $\calS$ is equal to $\sigma(\calS^\mo)$. For the second part, we refer to \cite[Lemma 4.1]{CheegerTwisted}. 
\end{proof}

The following result is obtained from \cite[Theorem 1.4]{BiswasCheegerCayley} (cf. \cite[Theorem 2.11]{CheegerCayleySum}, \cite[Theorem 2.6]{CayleyBottomBipartite}), \cite[Theorem 1.3]{CheegerCayleySum}, \cite[Theorem 1.1]{CheegerTwisted}. 

\begin{theorem}
\label{Thm:CombiToSpecForOne}
If $X$ denotes the Cayley graph $C(G, \calS)$ (resp. the Cayley sum graph $C_\Sigma(G, \calS)$, the twisted Cayley graph $C(G, \calS)^\sigma$, the twisted Cayley sum graph $C_\Sigma(G, \calS)^\sigma$) and $X$ is undirected and connected and is an $\varepsilon$-vertex expander for some $\varepsilon>0$, then the nontrivial spectrum of the adjacency operator of $X$ lies in the interval
$$\left( -1 + \frac{\varepsilon^4}{2^{9+\eta(X)}d^8} 
,
1 - \frac{\varepsilon^2}{2d^2}
\right]$$
where $\eta(X) = 0$ if $X= C(G, \calS), C_\Sigma(G, \calS)$ and $\eta(X) = 3$ if $X= C(G, \calS)^\sigma, C_\Sigma(G, \calS)^\sigma$.  
\end{theorem}

\begin{theorem}
\label{Thm:EquivExpansion}
Let $\sigma$ be an order two automorphism of a group $G$ and $|G|\geq 4$. Let $\calS$ be a multi-subset of $G$ of size $d$. Let $X_1, X_2, X_3, X_4$ denote the Cayley graph $C(G, \calS)$, the Cayley sum graph $C_\Sigma(G, \calS)$, the twisted Cayley graph $C(G, \calS)^\sigma$ and the twisted Cayley sum graph $C_\Sigma(G, \calS)^\sigma$ respectively. Suppose $X_i, X_j$ are undirected. In addition, assume that $S$ is symmetric if $X_i, X_j$ are equal to $C_\Sigma(G, \calS), C(G, \calS)^\sigma$ in some order. Then the following statements hold. 
\begin{enumerate}
\item 
Up to factors of $\pm 1$, the eigenvalues of the adjacency operator of $X_i$ are equal to the eigenvalues of the adjacency operator of $X_j$. More precisely, there are multi-subsets $\calE_i^+, \calE_i^-$ (resp. $\calE_j^+, \calE_j^-$) of the spectrum of $X_i$ (resp. $X_j$) such that the union of $\calE_i^+, \calE_i^-$ (resp. $\calE_j^+, \calE_j^-$) is equal to the spectrum of $X_i$ (resp. $X_j$), and $\calE_i^+ = \calE_j^+, \calE_i^- = -\calE_j^-$ as multisets, and the size of $\calE_i^+$ is equal to $(|G|+ f_{ij})/2$ where $f_{ij}$ denotes the number of solutions of $g = g_{ij}$ in $G$
and 
$$
g_{ij} = 
\begin{cases}
g & \text{ if } i = j ,\\
{g^\mo} & \text{ if } (i,j) = (1, 2), (2, 1), \\
{\sigma(g)} & \text{ if } (i,j) = (1, 3), (3,1),\\
{\sigma(g)^\mo} & \text{ if } (i,j) = (1, 4), (4,1), \\
{\sigma(g)^\mo} & \text{ if } (i,j) = (2, 3), (3, 2), \\
{\sigma(g)} & \text{ if } (i,j) = (2, 4), (4,2), \\
{g^\mo} & \text{ if } (i,j) = (3, 4), (4,3).
\end{cases}
$$
\item 
The graph $X_i$ is connected if and only if $X_j$ is connected. 
\item 
The graph $X_i$ is connected and bipartite if and only if $X_j$ is connected and bipartite. 
\item 
The graph $X_i$ is a two-sided $\varepsilon$-expander if and only if $X_j$ is so. 
\item If the graphs $X_i, X_j$ are connected, then their diameters satisfy
\begin{equation}
\label{Eqn:DiamRelationXY}
\diam(X_i) \leq 2 \diam(X_j), \quad \diam(X_j) \leq 2 \diam(X_i).
\end{equation}
\item 
If $X_i$ is an $\varepsilon$-vertex expander with $\varepsilon>0$, then $X_j$ is a two-sided $\delta$-expander with $\delta = \delta _{\varepsilon, d, X_i} = \frac{\varepsilon^4}{2^{9 + \eta(X_i)}d^8} > 0$
where 
$$
\eta(X_i) 
= 
\begin{cases}
0 & \text{ if } X_i = C(G, S), C_\Sigma(G, S), \\
3 & \text{ if } X_i = C(G, S)^\sigma, C_\Sigma(G, S)^\sigma .
\end{cases}
$$
\end{enumerate}
\end{theorem}

\begin{proof}
Denote the adjacency operators of $X_1, X_2, X_3, X_4$ by $A_1, A_2, A_3, A_4$ respectively. Let $P_{ij}$ denote the operator defined on the space $L^2(G)$ of $\mathbb C$-valued functions on $G$ which sends $\delta_g$ to $\delta_{g_{ij}}$. Note that $P_{ij}$ is either the identity map or it induces a permutation of order two on the set of functions $\delta_g$ with $g\in G$. Suppose $i, j$ are integers between $1$ and $4$ such that the graphs $X_i, X_j$ are undirected, and assume in addition that $\calS$ is symmetric if $X_i, X_j$ are equal to $C_\Sigma(G, \calS), C(G, \calS)^\sigma$ in some order. From Lemma \ref{Lemma:CayleyTwistedUndir}, it follows that the adjacency operators of $X_i, X_j$ satisfy $A_j = P_{ij} A_i$. By Lemma \ref{Lemma:EigenvalueXY}, up to factors of $\pm 1$, the eigenvalues of the adjacency operator of $X_i$ are equal to the eigenvalues of the adjacency operator of $X_j$. 

If $L^2(G)^\pm$ denote the $\pm$-eigenspaces of $P_{ij}$, i.e., if 
$$L^2(G) ^\pm = \frac 12(1 \pm P_{ij}) L^2(G),$$
then $A_j = A_i$ on $L^2(G)^+$ and $A_j = -A_i$ on $L^2(G)^-$. This shows that there are multi-subsets $\calE_i^+, \calE_i^-$ (resp. $\calE_j^+, \calE_j^-$) of the spectrum of $X_i$ (resp. $X_j$) such that the union of $\calE_i^+, \calE_i^-$ (resp. $\calE_j^+, \calE_j^-$) is equal to the spectrum of $X_i$ (resp. $X_j$), and $\calE_i^+ = \calE_j^+, \calE_i^- = -\calE_j^-$ as multisets, and the size of $\calE_i^+$ is equal to $(|G|+ f_{ij})/2$.

Note that $P_{ij}$ fixes $\delta_e$ where $e$ denotes the identity element of $G$. By Lemma \ref{Lemma:EigenvalueXY}, one of them is connected if and only if so is the other, one of them is connected and bipartite if and only if so is the other. Since the eigenvalues of $X_i$ are equal to the eigenvalues of $X_j$ up to factors of $\pm 1$, it follows that if the nontrivial eigenvalues of $X_i$ lie in the interval $(-1+\varepsilon, 1-\varepsilon)$, then the nontrivial eigenvalues of $X_j$ also lie in the same interval. The bounds on the diameters of $X_i, X_j$ follows from Lemma \ref{Lemma:EigenvalueXY}.

Suppose $X_i$ is an $\varepsilon$-vertex expander with $\varepsilon>0$. Then by Theorem \ref{Thm:CombiToSpecForOne}, it follows that the nontrivial eigenvalues of $X_i$ lie in the interval $(-1+\delta_{\varepsilon, d, X_i}, 1 - \delta_{\varepsilon, d, X_i})$. Hence, the nontrivial eigenvalues of $X_j$ also lie in the interval $(-1 + \delta_{\varepsilon, d, X_i}, 1 - \delta_{\varepsilon, d, X_i})$, i.e., $X_j$ is also a two-sided $\delta_{\varepsilon, d, X_i}$-expander.
\end{proof}

Before proceeding further, let us point out the following remark.

\begin{remark} 
\label{Rk:NeedInvolution}
Let $X_i, X_j$ be as in Theorem \ref{Thm:EquivExpansion}. Assume that they are undirected and $\calS$ is symmetric if these graphs are equal to $C_\Sigma(G, \calS), C(G, \calS)^\sigma$ in some order. Note that the square graphs of $X_i, X_j$ are equal, where the square graph of a graph $\Gamma$ is the graph having the same set of vertices and its adjacency operator is equal to the square of the adjacency operator of $\Gamma$. This shows that the eigenvalues of $X_i, X_j$ are equal up to factors of $\pm 1$ and from there, it can be deduced that one of them is connected and non-bipartite if and only if so is the other. Further, if one of them is a two-sided $\delta$-expander if and only if so is the other. However, the other parts of Theorem \ref{Thm:EquivExpansion} do not follow from the above observation.

Indeed, Theorem \ref{Thm:EquivExpansion} provides a more refined understanding of the relationship between the spectra of the graphs $X_i, X_j$, than just using the above-mentioned technique of squaring the graph. For instance, the diameter bounds are not deducible from the equality of the square graphs. Just using $A_i^2 = A_j^2$ and assuming that all the entries of $\sum_{r = 1}^d A_i^r$ are positive for some integer $d\geq 1$ (i.e., $\diam(X_i) \leq d$ for some $d\geq 1$), there does not seem to be a way to obtain an integer $d'\geq 1$ such that all the entries of $\sum_{r = 1}^{d'} A_j^r$ are positive, i.e., to obtain an upper bound on the diameter of $X_j$. In other words, no bound on the diameter of one of the two graphs $X_i, X_j$ seem to follow from a bound on the diameter of the other from the relation $A_i^2 = A_j^2$ alone. An important property of the involution $P_{ij}$, that plays a crucial role in the proof of Theorem \ref{Thm:EquivExpansion}, is that it fixes the element $\delta_e$ of $L^2(G)$ for any $i, j$, and thus Lemma \ref{Lemma:EigenvalueXY} can be applied to obtain the bounds on the diameters of $X_i, X_j$ as stated in Equation \eqref{Eqn:DiamRelationXY}. 
\end{remark}

We contrast the following corollary of Theorem \ref{Thm:EquivExpansion} with the results of Lev \cite[Proposition 1]{LevSumDiffHamiltCycle}, and Amooshahi--Taeri \cite[Theorems 1, 2]{AmooshahiTaeri}, who studied the connectedness of the Cayley sum graph $C_\Sigma(G, S)$ without the supplementary hypothesis that $S$ is symmetric. 

\begin{corollary}
Let $X$ be a graph among $C_\Sigma(G, S), C(G, S)^\sigma, C_\Sigma(G,  S)^\sigma$. Suppose $X$ is undirected. If $S$ is symmetric, then $X$ is connected if and only if $S$ generates $G$. 
\end{corollary}

\section{Eigenvalues of Cayley sum graphs and the twisted variants}

Let $\rho_1, \ldots,\rho_r$ denote the irreducible representations of $G$ over $\bbC$ and let $\chi_1, \ldots, \chi_r$ denote their characters. For a symmetric subset $S$ of $G$ closed under conjugation, the eigenvalues of the adjacency operator of the Cayley graph $C(G, S)$ of $G$ are equal to 
$$\frac 1{\dim \rho_i} \sum_{s\in S} \chi_i(s),$$
which occurs with multiplicity $\dim \rho_i$ for $1\leq i \leq r$ (see the works of Babai \cite{BabaiSpectraofCayley}, Diaconis--Shahshahani \cite{DiaconisShahshahaniGeneRandPermTrans}). The following result expresses the eigenvalues of Cayley sum graphs and the twisted variants through the values of the characters of the underlying group, up to factors of $\pm 1$. This serves as an analogue of the above result for the Cayley graphs, and thus proves a sought after result for the Cayley sum graphs \cite[p. 112]{LubotzkyDiscreteGroups}. 

\begin{theorem}
\label{Thm:SpectrumThruChar}
Suppose $S$ is symmetric and closed under conjugation. 
\begin{enumerate}
\item 
Up to factors of $\pm 1$, the eigenvalues of the adjacency operator of the Cayley sum graph $C_\Sigma(G, S)$ are equal to 
$$\frac 1{\dim \rho_i} \sum_{s\in S} \chi_i(s),$$
which occurs with multiplicity $\dim \rho_i$ for $1\leq i \leq r$. 

\item 
If $\sigma(S) = S$, then up to factors of $\pm 1$, the eigenvalues of the adjacency operator of the twisted Cayley graph $C(G, S)^\sigma$ are equal to 
$$\frac 1{\dim \rho_i} \sum_{s\in S} \chi_i(s),$$
which occurs with multiplicity $\dim \rho_i$ for $1\leq i \leq r$. 

\item 
If $\sigma(S) = S$, then up to factors of $\pm 1$, the eigenvalues of the adjacency operator of the twisted Cayley sum graph $C_\Sigma(G, S)^\sigma$ are equal to 
$$\frac 1{\dim \rho_i} \sum_{s\in S} \chi_i(s),$$
which occurs with multiplicity $\dim \rho_i$ for $1\leq i \leq r$. 
\end{enumerate}
Moreover, the number of factors of $1$ appearing in statement (1) (resp. (2), (3)) is equal to the average of $|G|$ and the number of solutions of $g = g^\mo$ (resp. $g = \sigma(g), g = \sigma(g)^\mo$) in $G$. 
\end{theorem}

The above result follows from Theorem \ref{Thm:EquivExpansion}. 

\section{Distinguishing twisted Cayley graphs}
\label{Sec:TwistedCayley}

In the subsequent sections, we obtain several results involving twisted Cayley graphs. Therefore, it is important to be able to distinguish them from Cayley graphs. In this section, we describe certain properties of twisted Cayley graphs. More specifically, we focus on the following questions. 
\begin{enumerate}
\item 
Under what conditions, a twisted Cayley graph $C(G, S)^\sigma$ is non-isomorphic to the Cayley graph $C(G, S)$? 
\item 
Under what conditions, the twisted Cayley graphs $C(G, S)^\sigma$ are pairwise non-isomorphic when $\sigma$ varies over certain order two automorphisms of $G$?
\end{enumerate}

Note that the set of fixed points of an order two automorphism $\sigma$ of $G$ forms a subgroup $G^\sigma$ of $G$, and it acts on the twisted Cayley graph $C(G, S)^\sigma$ as follows. For $g\in G^\sigma$, define $g\cdot v = gv$ and $g\cdot \{u, v\} = \{gu, gv\}$ for any vertex $v$ and for any edge $\{u, v\}$ of $C(G, S)^\sigma$. Thus, $G^\sigma$ is a subgroup of the group of automorphisms of $C(G, S)^\sigma$. 

Given a subset $S$ of $G$ and an automorphism $\sigma$ of $G$, let $G^{S, \sigma}$ denote the set of elements $g\in G$ such that the left multiplication by $g$ map induces an automorphism on $C(G, S)^\sigma$, i.e., the map 
$$v\mapsto gv ,\{u, v\}\mapsto \{gu, gv\}$$
is an automorphism of $C(G, S)^\sigma$. Note that $G^{S, \sigma}$ is a subgroup of $G$, and it contains $G^\sigma$ as a subgroup. 

In certain cases, $G^{S, \sigma}$ could be as big as $G$ (for instance, when $S = G$). However, if $S$ is small (for instance, when $S$ contains the identity element only), then $G^{S, \sigma}$ is as small as possible, in fact, it is equal to $G^\sigma$. 

\begin{lemma}
\label{Lemma:TwoSubgroupsAreEqual}
If the size of $SS^\mo$ is smaller than the minimum of the sizes of the nontrivial conjugacy classes in $G$, then $G^{S, \sigma}$ is equal to the subgroup $G^\sigma$ of $G$ consisting of the fixed points of $\sigma$. 
\end{lemma}

\begin{proof}
Let $g$ be an element of $G^{S, \sigma}$. For every $x \in G, s\in S$, the set $\{gx , g\sigma(xs)\}$ is an edge in $C(G, S)^\sigma$. So, there is an element $t$ in $S$ such that $g\sigma(xs) = \sigma(gxt)$ holds, which implies $st^\mo = x^\mo \sigma(g^\mo) gx$, and hence the conjugacy class of $\sigma(g^\mo)g$ is contained in $SS^\mo$. By the assumption on the size of $SS^\mo$, it follows that $\sigma(g) = g$. This shows that $G^{S, \sigma}$ is contained in $G^\sigma$. This completes the proof. 
\end{proof}

To be able to compare the twisted Cayley graphs $C(G, S)^\sigma$ when $\sigma$ runs over certain automorphisms, it would be important to be able to attach a useful invariant to twisted Cayley graphs. Note that from the triple $(G, S, \sigma)$, one obtains the twisted Cayley graph $C(G, S)^\sigma$ and one also obtains $G^{S, \sigma}$. We could ask that whether the group $G^{S, \sigma}$ could be thought as an invariant of $C(G, S)^\sigma$. 

\begin{question}
\label{Qn:InvariantAssociation}
Is there a way to obtain $G^{S, \sigma}$ from the graph $C(G, S)^\sigma$ alone? 
\end{question}

\subsection{A sufficient condition for two twisted Cayley graphs to be isomorphic}
Next, we provide a sufficient criterion for showing that two twisted Cayley graphs are isomorphic. Let $\sigma, \tau$ denote inner automorphisms of $G$ corresponding to the elements $\tilde \sigma, \tilde \tau$ of $G$ of order two. If $C(G, S)^\sigma, C(G, S)^\tau$ are isomorphic, then assuming that Question \ref{Qn:InvariantAssociation} admits an answer in the affirmative (possibly under certain conditions on $S$), it follows that $G^{S, \sigma}, G^{S, \tau}$ are isomorphic, which implies that $G^\sigma, G^\tau$ are isomorphic (under the hypothesis that $S$ is `small' so that Lemma \ref{Lemma:TwoSubgroupsAreEqual} is applicable). Thus, roughly speaking, when $C(G, S)^\sigma, C(G, S)^\tau$ are isomorphic, one may have that $G^\sigma, G^\tau$ are isomorphic. On the other hand, if $\tilde \sigma, \tilde \tau$ are conjugates, then it follows that $G^\sigma, G^\tau$ are isomorphic. Thus, one may wonder whether it follows that $C(G, S)^\sigma, C(G, S)^\tau$ are isomorphic when $\tilde \sigma, \tilde \tau$ are conjugates. The following Proposition confirms that this is the case under suitable hypothesis. 

In the following, for $? = \emptyset$, the graph $C_?(G, S)^\sigma$ is to be understood as the twisted Cayley graph $C(G, S)^\sigma$. 

\begin{proposition}
\label{Prop:IsomTwistedCayley}
Let $S$ a subset of a finite group $G$, and $\sigma, \tau$ be inner automorphisms of $G$ induced by the elements $\tilde \sigma, \tilde \tau$ of $G$. Let $G$ be a subgroup of a group $\calG$ and there exists an element $g\in \calG$ such that $g \tilde \sigma g^\mo= \tilde \tau$ holds in $\calG$. If $g Sg^\mo = S$ and $gGg^\mo = G$, then the graphs $C_?(G, S)^\sigma, C_?(G, S)^\tau$ are isomorphic for $? \in \{\emptyset, \Sigma\}$. In particular, if $\sigma, \tau$ are inner automorphisms of $\frakS_n$ corresponding to the elements $\tilde \sigma, \tilde \tau$ of $\frakS_n$, and if the elements $\tilde \sigma, \tilde \tau$ lie in the same conjugacy class, then for any subset $S$ of $\frakS_n$, the twisted Cayley graphs 
$$C_?(\frakS_n, \cup_i g^iSg^{-i})^\sigma, C_?(\frakS_n, \cup_i g^iSg^{-i})^\tau$$
are isomorphic for some $g\in G$, and for $? \in \{\emptyset, \Sigma\}$.
\end{proposition}

\begin{proof}
Let $? \in \{\emptyset, \Sigma\}$, and 
$$
\delta 
= 
\begin{cases}
1 & \text{ if } ? = \emptyset, \\
-1 & \text{ if } ? = \Sigma.
\end{cases}
$$
Note that if $\{u, v\}$ is an edge in $C_?(G, S)^\sigma$, i.e., $v = \sigma(u^\delta s)$ for some $s\in S$, then 
\begin{align*}
g v g^\mo 
& = g (\tilde \sigma u^\delta s \tilde \sigma^\mo) g^\mo \\
& = g \tilde \sigma g^\mo g u^\delta g^\mo g s \tilde g^\mo g \sigma^\mo g^\mo \\
& = \tilde \tau gu^\delta g^\mo gsg^\mo \tilde \tau^\delta \\
& = \tau (gu^\delta g^\mo gsg^\mo),
\end{align*}
which shows that $gvg^\mo$ is adjacent to $gug^\mo$ in $C_?(G, S)^\tau$. This shows that 
$$u \mapsto gug^\mo, \{u, v\} \mapsto \{gug^\mo, gvg^\mo\}$$ 
defines a map from $C_?(G, S)^\sigma$ to $C_?(G, S)^\tau$. It follows that this map is an isomorphism of graphs. 

The second part follows from the first part.
\end{proof}

\begin{remark}
An upshot of the preceding discussion is that considering the twisted Cayley graphs of a group $G$ with respect to inner automorphisms corresponding to two (or more) elements of $G$ lying in distinct conjugacy classes could be a way to obtain examples of graphs, which are pairwise non-isomorphic. Using this guiding principle, we consider the twisted Cayley graphs of permutation groups with respect to certain inner automorphisms, and prove that that they are non-isomorphic (see Theorem \ref{Thm:AsymptoticallyIsoSpecNonIsom}). 
\end{remark}

Proposition \ref{Prop:IsomTwistedCayley} shows that given a normal subset $S$ of a group $G$, the total number of distinct isomorphism classes of the twisted Cayley graphs $C(G, S)^\sigma$ corresponding to the inner automorphisms $\sigma$ of $G$, is bounded by the number of conjugacy classes of $G$ from the above. Note that this bound could be too weak in certain cases, for instance, when $S$ is too large ($S= G$, for example). If we take $S$ to be a small subset (for instance, the set of permutations of a specific cycle type if $G = \frakS_n$ or $\frakA_n$), then this bound may get closer to an equality. This suggests the following question. 

\begin{question}
Given a group $G$, determine all the normal subsets $S$ of $G$ such that the number of distinct isomorphism classes of the twisted Cayley graphs $C(G, S)^\sigma$ corresponding to the inner automorphisms $\sigma$ of $G$ is equal to (or close to) the number of conjugacy classes in $G$. 
\end{question}

It has been proved by Roichman that the Cayley graphs of symmetric groups (resp. alternating groups) with respect to the conjugacy class of long odd (resp. even) cycles are expanders, and such graphs with respect to the conjugacy class of short cycles are non-expanders \cite{RoichmanExpansionCayleyAlt}. Thus, it would be interesting to answer the above question for symmetric groups and alternating groups, where the connection set $S$ consists of permutations of certain cycle types. 

\subsection{Spectrum admits a symmetric structure or a twisted Cayley graph is non-isomorphic}

If $\sigma, \tau$ are inner automorphisms of $G$ induced by the elements $\tilde \sigma, \tilde \tau$ of $G$, then Proposition \ref{Prop:IsomTwistedCayley} asserts that $C(G, S)^\sigma, C(G, S)^\tau$ are isomorphic if $gSg^\mo = S$ holds for some element $g\in G$ satisfying $g \tilde \sigma g^\mo= \tilde \tau$. It may happen that $gSg^\mo = S$ does not hold for any element $g\in G$ satisfying $g \tilde \sigma g^\mo= \tilde \tau$. For instance, let $\sigma, \tau$ denote the inner automorphisms of $\frakS_n$ induced by the elements $(1, 2), (3, 4)$ respectively, and let $S$ denote the subset 
$$
\{(1, 2), (3, 4), 
(1, 2, 3, 4, \ldots, n), 
(2, 1, 3, 4, \ldots, n), 
(1, 2, 4, 3, \ldots, n), 
(2, 1, 4, 3, \ldots, n)
\}$$
of $\frakS_n$. Note that $gSg^\mo \neq S$ for any $g\in \frakS_n$ with $g(1, 2)g^\mo = (3, 4)$. 

On the other hand, distinguishing the Cayley graph $C(\frakS_n, S)$ from the twisted Cayley graphs $C(\frakS_n, S)^\sigma, C(\frakS_n, S)^\tau$ does not seem to be immediate. However, if the spectrum of the Cayley graph $C(\frakS_n, S)$ does not contain a large subset (having size $\frac 12 (n! - 2(n-2)!)$, i.e., having size a little less than $50\%$ of that of the spectrum) symmetric about the origin, then it turns out that the Cayley graph $C(\frakS_n, S)$ is non-isomorphic to both of $C(\frakS_n, S)^\sigma, C(\frakS_n, S)^\tau$, and further, if the spectrum of the Cayley graph $C(\frakS_n, S)$ does not contain a large subset (having size a little less than $75\%$ of that of the spectrum) symmetric about the origin, then it turns out that the Cayley graph $C(\frakS_n, S)$ is non-isomorphic to one of $C(\frakS_n, S)^\sigma, C(\frakS_n, S)^\tau, C(\frakS_n, S)^{\sigma \tau}$ (Proposition \ref{Prop:2PowerKSymm}). This suggests that in the presence of a large (yet proper) symmetric subset in the spectra of a Cayley graph $C(G, S)$, one needs to consider its twist by a `random enough' order two automorphism in order to obtain a twisted Cayley graph, not isomorphic to the Cayley graph $C(G, S)$. This guiding principle is justified by Theorem \ref{Thm:75Symm}, which is established using Proposition \ref{Prop:2PowerK} below. Theorem \ref{Thm:75Symm} is illustrated in the context of symmetric groups by proving that for any given $k\geq 1$, the spectrum of the Cayley graph of the symmetric group $\frakS_n$ (with respect to a connection set $S_{n,k}$, similar to the set $S$ considered above) contains a symmetric subset of size around $(1 - \frac 1{2^k} - o(1))n!$ or this Cayley graph is non-isomorphic to the twisted Cayley graph of $\frakS_n$ with respect to some order two inner automorphism of $\frakS_n$ (see Proposition \ref{Prop:2PowerKSymm}). 

\begin{proposition}
\label{Prop:2PowerK}
Let $k$ be a positive integer, and $P_1, \ldots, P_k$ be pairwise commuting operators on a vector space $V$, each having order $\leq 2$. For any $\varepsilon = (\varepsilon_1, \ldots, \varepsilon_k) \in \{\pm 1\}^k$, let $V_\varepsilon$ denote the subspace 
$$\prod_{i=1}^k (1 + \varepsilon_i P_i) V$$
of $V$. Let $A$ be an operator on $V$ and assume that it commutes with $P_1, \ldots, P_k$. For $\varepsilon \in \{\pm 1\}^k$, let $\calE_{\varepsilon, +}$ (resp. $\calE_{\varepsilon, -}$) denote the multi-set consisting of the positive (resp. negative) characteristic roots of $A|_{V_\varepsilon}: V_\varepsilon \to V_\varepsilon$. If the spectrum of $A$ coincides with the spectrum of $(\prod_{i\in I} P_i ) A$ for any nonempty subset $I$ of $\{1, 2, \ldots, k\}$, then 
$$\calE_{\varepsilon, +} = - \calE_{\varepsilon, -}$$
for any $\varepsilon\in \{\pm 1\}^k$ with $\varepsilon \neq (1, 1, \ldots, 1)$, and the spectrum of $A$ contains a subset of size 
$$= \dim V - \dim V_{(1, 1, \ldots, 1)},$$
which is symmetric about the origin. 
\end{proposition}

\begin{proof}
We apply induction to prove the result. Assume that $k=1$. Note that the multiset of positive eigenvalues of $A$ is equal to $\calE_{1, +} \cup \calE_{-1, +}$ and the multiset of positive eigenvalues of $P_1A$ is equal to $\calE_{1, +} \cup (-\calE_{-1, -})$. If $A, P_1A$ have the same spectrum, then it follows that $\calE_{-1, +} = -\calE_{-1, -}$. 

Assume that $k=2$. Note that the multisets of positive eigenvalues of $A, P_1A, P_2A, P_1P_2A$ are equal to 
$$\calE_{(1,1), +} \cup \calE_{(1,-1), +} \cup \calE_{(-1,1), +} \cup \calE_{(-1,-1), +} ,
$$
$$\calE_{(1,1), +} \cup \calE_{(1,-1), +} \cup (-\calE_{(-1,1), -}) \cup (-\calE_{(-1,-1), -}) ,
$$
$$\calE_{(1,1), +} \cup (-\calE_{(1,-1), -}) \cup \calE_{(-1,1), +} \cup (-\calE_{(-1,-1), -}) ,
$$
$$\calE_{(1,1), +} \cup (-\calE_{(1,-1), -}) \cup (-\calE_{(-1,1), -}) \cup \calE_{(-1,-1), +} .
$$
Equating the multiset of positive eigenvalues of $A$ with those of $P_1A, P_2A, P_1P_2A$, we obtain 
$$\calE_{(-1,1), +} \cup \calE_{(-1,-1), +} 
=
(-\calE_{(-1,1), -}) \cup (-\calE_{(-1,-1), -}) ,
$$
$$ 
\calE_{(1,-1), +} \cup \calE_{(-1,-1), +} 
=
(-\calE_{(1,-1), -}) \cup (-\calE_{(-1,-1), -}) ,
$$
$$\calE_{(1,-1), +} \cup \calE_{(-1,1), +}
=
(-\calE_{(1,-1), -}) \cup (-\calE_{(-1,1), -}).
$$
Combining the above three equations, it follows that 
$\calE_{(1,-1), +}
=
-\calE_{(1,-1), -} $,
and hence 
$\calE_{(-1,1), +}
=
-\calE_{(-1,1), -},
$
and 
$$\calE_{(-1,-1), +} 
=
-\calE_{(-1,-1), -}
$$
hold. This proves the result for $k=2$. 

Assume that the result holds for some integer $k\geq 2$. Let $P_1, \ldots, P_{k+1}$ be pairwise commuting operators on a vector space $V$ and $A$ be an operator on $V$ that commutes with $P_i$ for any $1\leq i \leq k+1$. We need to show that 
$$\calE_{\varepsilon, +} = - \calE_{\varepsilon, -}$$
for any $\varepsilon\in \{\pm 1\}^{k+1}$ with $\varepsilon \neq (1, 1, \ldots, 1)$. 

In the following, for any $\varepsilon\in \{\pm 1\}^{k+1}$, and for any $1\leq a, b \leq k+1$, we denote the element of $\{\pm 1\}^{k+1}$ obtained from $\varepsilon$ by changing the sign of its $a$-th coordinate (resp. $a$-th and the $b$-th coordinates) by $\varepsilon_{(a)}$ (resp. $\varepsilon_{(a, b)}$.

Let $\varepsilon = (\varepsilon_1, \ldots, \varepsilon_{k+1})$ be an element of $\{\pm 1\}^{k+1}$ with $\varepsilon \neq (1, 1, \ldots, 1)$. Consider the case when $\varepsilon_i = \varepsilon_j = 1$ holds for at least two distinct $i,j$. Note that 
$
\calE_{\varepsilon, +} \cup 
\calE_{\varepsilon_{(i)}, +}$ 
(resp. $
\calE_{\varepsilon, -} \cup 
\calE_{\varepsilon_{(i)}, -}
$) 
is equal to the multi-set of positive (resp. negative) characteristic roots of the restriction of $A$ on the space
$$\prod_{1\leq \ell \leq k+1, \ell \neq i} (1 + \varepsilon_\ell P_\ell)V.$$ 
Since $\varepsilon \neq (1, 1, \ldots, 1)$ and $k+1\geq 3$, it follows that the element of $\{\pm 1\}^k$ obtained from $\varepsilon$ by removing its $i$-th coordinate is not equal to the element $(1, \ldots, 1)$ of $\{\pm 1\}^k$. By the induction hypothesis, it follows that 
$$
\calE_{\varepsilon, +} \cup 
\calE_{\varepsilon_{(i)}, +}
=
(-\calE_{\varepsilon, -}) \cup 
(-\calE_{\varepsilon_{(i)}, -}).
$$
Note that 
$
\calE_{\varepsilon, +} \cup 
\calE_{\varepsilon_{(i), (j)}, +}$ 
(resp. $
\calE_{\varepsilon, -} \cup 
\calE_{\varepsilon_{(i), (j)}, -}
$) 
is equal to the multi-set of positive (resp. negative) characteristic roots of the restriction of $A$ on the space
$$(1 + P_iP_j)\prod_{1\leq \ell \leq k+1, \ell \neq i,j} (1 + \varepsilon_\ell P_\ell)V.$$ 
Since $\varepsilon \neq (1, 1, \ldots, 1)$ and $k+1\geq 3$, it follows that $\varepsilon_\ell = -1$ for some $1\leq \ell \leq k+1$ satisfying $\ell \neq i, j$. By the induction hypothesis, it follows that
$$
\calE_{\varepsilon, +} \cup 
\calE_{\varepsilon_{(i), (j)}, +}
= 
(-\calE_{\varepsilon, -}) \cup 
(-\calE_{\varepsilon_{(i), (j)}, -}).
$$
Moreover, the set 
$
\calE_{\varepsilon_{(i)}, +}
\cup 
\calE_{\varepsilon_{(i), (j)}, +}
$
(resp. $
\calE_{\varepsilon_{(i)}, -}
\cup 
\calE_{\varepsilon_{(i), (j)}, -}$)
is equal to the multi-set of positive (resp. negative) characteristic roots of the restriction of $A$ on the space
$$(1 - P_i)\prod_{1\leq \ell \leq k+1, \ell \neq i,j} (1 + \varepsilon_\ell P_\ell)V.$$ 
Applying the induction hypothesis, we obtain
$$
\calE_{\varepsilon_{(i)}, +}
\cup 
\calE_{\varepsilon_{(i), (j)}, +}
= 
(-\calE_{\varepsilon_{(i)}, -})
\cup 
(-\calE_{\varepsilon_{(i), (j)}, -}).
$$
It follows that 
$$\calE_{\varepsilon, +} = - \calE_{\varepsilon, -}$$
for any $\varepsilon\in \{\pm 1\}^{k+1}$ having at least two distinct coordinates equal to $1$ with $\varepsilon \neq (1, 1, \ldots, 1)$. 

Let $\varepsilon = (\varepsilon_1, \ldots, \varepsilon_{k+1})$ be an element of $\{\pm 1\}^{k+1}$ having exactly one coordinate positive. Let $i$ be such that $\varepsilon_i=1$. 
Let $1\leq j \leq k+1$ be an integer with $j \neq i$. It follows that 
$$
\calE_{\varepsilon_{(j)}, +}
=
-\calE_{\varepsilon_{(j)}, -}.
$$
Note that 
$
\calE_{\varepsilon, +} \cup 
\calE_{\varepsilon_{(j)}, +}$ 
(resp. $
\calE_{\varepsilon, -} \cup 
\calE_{\varepsilon_{(j)}, -}
$) 
is equal to the multi-set of positive (resp. negative) characteristic roots of the restriction of $A$ on the space
$$\prod_{1\leq \ell \leq k+1, \ell \neq j} (1 + \varepsilon_\ell P_\ell)V.$$ 
Since $\varepsilon \neq (1, 1, \ldots, 1)$ and $k+1\geq 3$, it follows that the element of $\{\pm 1\}^k$ obtained from $\varepsilon$ by removing its $j$-th coordinate is not equal to the element $(1, \ldots, 1)$ of $\{\pm 1\}^k$. By the induction hypothesis, it follows that 
$$
\calE_{\varepsilon, +} \cup 
\calE_{\varepsilon_{(j)}, +}
=
(-\calE_{\varepsilon, -}) \cup 
(-\calE_{\varepsilon_{(j)}, -}),
$$
and hence 
$$
\calE_{\varepsilon, +} 
=
-\calE_{\varepsilon, -}.
$$

Let $\varepsilon$ denote the element $(-1, \ldots, -1)$ of $\{\pm 1\}^{k+1}$. Note that 
$$
\calE_{\varepsilon_{(1)}, +} 
=
-\calE_{\varepsilon_{(1)}, -}.
$$

Note that 
$
\calE_{\varepsilon, +} \cup 
\calE_{\varepsilon_{(1)}, +}$ 
(resp. $
\calE_{\varepsilon, -} \cup 
\calE_{\varepsilon_{(1)}, -}
$) 
is equal to the multi-set of positive (resp. negative) characteristic roots of the restriction of $A$ on the space
$$\prod_{1\leq \ell \leq k+1, \ell \neq 1} (1 - P_\ell)V.$$ 
By the induction hypothesis, it follows that 
$$
\calE_{\varepsilon, +} \cup 
\calE_{\varepsilon_{(1)}, +}
=
(-\calE_{\varepsilon, -}) \cup 
(-\calE_{\varepsilon_{(1)}, -}),
$$
and hence 
$$
\calE_{\varepsilon, +} 
=
-\calE_{\varepsilon, -}.
$$
This completes the proof of the statement that
$$\calE_{\varepsilon, +} = - \calE_{\varepsilon, -}$$
holds for any $\varepsilon\in \{\pm 1\}^{k+1}$ with $\varepsilon \neq (1, 1, \ldots, 1)$. 

For $\varepsilon\in \{\pm 1\}^k$, let $\calE_{\varepsilon, 0}$ denote the multi-set consisting of the characteristic roots of $A|_{V_\varepsilon}: V_\varepsilon \to V_\varepsilon$ which are equal to $0$. If the spectrum of $A$ coincides with the spectrum of $(\prod_{i\in I} P_i ) A$ for any nonempty subset $I$ of $\{1, 2, \ldots, k\}$, then the spectrum of $A$ contains
$$
\cup_{\varepsilon\in \{\pm 1\}^k, \varepsilon\neq (1, \ldots, 1)} 
(\calE_{\varepsilon, +} \cup \calE_{\varepsilon, -} \cup \calE_{\varepsilon, 0} ),
$$
which is symmetric about the origin, and has size equal to 
$$
\sum_{\varepsilon\in \{\pm 1\}^k, \varepsilon\neq (1, \ldots, 1)} \dim V_\varepsilon= 
\dim V - \dim V_{(1, 1, \ldots, 1)}.$$
\end{proof}

\begin{theorem}
\label{Thm:75Symm}
Let $\sigma_1, \ldots, \sigma_k$ be pairwise commuting automorphisms of $G$ of order $\leq 2$. Assume that the Cayley graph $C(G, S)$ is undirected and so is the twisted Cayley graph $C(G, S)^{\sigma_i}$ for any $1\leq i \leq k$. Suppose the spectrum of $C(G, S)$ contains no subset of size 
$$ 
\dim L^2(G) - \dim L^2(G)^{\sigma_1 = 1, \sigma_2 = 1, \ldots, \sigma_k = 1}
,$$
which is symmetric about the origin. Then the graph $C(G, S)$ is non-isospectral to at least one of the twisted Cayley graphs $C(G, S)^{\prod_{i\in I} \sigma_i}$ where $I$ runs over nonempty subsets of $\{1, 2, \ldots, k\}$. In particular, the Cayley graph $C(G, S)$ is non-isomorphic to one among these twisted Cayley graphs. 
\end{theorem}

\begin{proof}
It follows from Proposition \ref{Prop:2PowerK} and the proof of Theorem \ref{Thm:EquivExpansion}.
\end{proof}

\begin{remark}
Roughly speaking, Theorem \ref{Thm:75Symm} suggests that the spectrum of an undirected Cayley graph $C(G, S)$ is `too symmetric', or one among its twists by involutions (pairwise commuting order two automorphisms of $G$ that yield undirected twisted Cayley graphs from $G$ with $S$ as the connection set) is non-isomorphic to $C(G, S)$.
In other words, we have the dichotomy that the spectrum of a Cayley graph $C(G, S)$ has a `symmetric structure', or one among its several twists is `random enough' to be non-isomorphic to $C(G, S)$.
\end{remark}

\subsection{Uniform actions, symmetry in spectrum and twisted Cayley graphs}

If each term of a sequence $\{V_n\}_{n\geq 1}$ of vector spaces carries an action of $H$ through linear automorphisms and $\lim\dim V_n = \infty$, then $H$ is said to act \textit{uniformly} on this sequence if for large enough $n$, the $H$-module $V_n$ is very close to being the direct sum of several copies of the regular representation $\bbC[H]$ of $H$, i.e., 
$$
\lim_{n\to \infty} 
\frac 
{\dim V_{n, \rho}}
{\dim V_n}
= \frac {(\dim \rho)^2}{|H|}
$$
holds for any irreducible representation $\rho$ of $H$. 

Note that for $m \leq n$, an element of $\frakS_m$ can be thought of as an element of $\frakS_n$, which fixes any integer $> m$. Let $k$ be a positive integer. Let $\frakS_k$ act on $\frakS_n$ through inner automorphisms for any $n\geq k$. For each $n\geq 2k$, let $T_{k,n}$ denote the subset of $\frakS_n$ defined by 
$$T_{k,n}
= \{(1, 2), (3, 4), (5, 6), \ldots, (2k-1, 2k)
\}.
$$
Let $H_{k,n}$ denote the subgroup of $\frakS_n$ generated by $T_{k,n}$. 

\begin{lemma}
\label{Lemma:UniformOnSn}
Let $H$ be a subgroup of $\frakS_k$, which acts on $\frakS_n$ and on $\frakA_n$ via inner automorphisms. Let $G_n$ denote $\frakS_n$ or $\frakA_n$. The action of $H$ on $L^2(G_n)$ is uniform as $n\to \infty$, i.e.,
$$
\lim_{n\to \infty} 
\frac 
{\dim L^2(G_{n})_{ \rho}}
{\dim L^2(G_{n})}
= \frac {(\dim \rho)^2}{|H|}
$$
holds for any irreducible representation $\rho$ of $H$. In particular, if $(\bbZ/2\bbZ)^k$ acts on $\frakS_n$ via the composite map 
$$H \to H_{k,n} \to \Aut(\frakS_n),$$
where the first map sends $e_i$ to $(2i-1, 2i)$ and the second map is the conjugation action of $H_{k,n}$ on $\frakS_n$, then its action on $L^2(\frakS_n)$ is uniform as $n\to \infty$, i.e.,
$$
\lim_{n\to \infty} 
\frac 
{\dim L^2(\frakS_{n})_{ \chi}}
{\dim L^2(\frakS_{n})}
= \frac 1{|(\bbZ/2\bbZ)^k|}
$$
holds for any character $\chi$ of $(\bbZ/2\bbZ)^k$. 
\end{lemma}

\begin{proof}
Let $m_{n, H}$ denote the number of elements $g \in G_n$, such that the conjugates of $g$ by the elements of $H$ are all distinct. 
We claim that $m_{n,H}$ is equal to $|G_n| - O((n-1)!)$. 

Let $\pi$ be an element of $\frakS_n$ such that $\pi(i) > 2k$ for any $1\leq i \leq 2k$. Let $h$ be a nontrivial element of $H$. Let $1\leq i\leq 2k$ be an integer such that $h(i)\neq i$. 
Note that the images of $i$ under $h\pi h^\mo, \pi$ are distinct. So, $h\pi h^\mo \neq \pi$ for any nontrivial element $h$ of $H$. 
By the inclusion-exclusion principle, the number of elements $\pi \in \frakS_n$ satisfying $\pi(i) > 2k$ for any $1\leq i \leq 2k$ is equal to 
$$n! + \sum_{r = 1}^{2k} (-1)^r\binom {2k} r (n-r)! .$$
It follows that $m_{n,H}$ is equal to $n! - O((n-1)!)$ when $G_n = \frakS_n$. 
Note that if $\pi$ is an element of $\frakS_n$ satisfying $\pi(i) > 2k$, then $\pi (2k+1, 2k+2)$ is also an element of $\frakS_n$ satisfying $\pi(i) > 2k$, and the permutations $\pi, \pi(2k+1, 2k+2)$ have different parity. Hence, 
the number of elements $\pi$ in $\frakA_n$ satisfying $\pi(i) > 2k$
is equal to half of the number of elements $\pi$ in $\frakS_n$ satisfying $\pi(i) > 2k$. This proves the claim that $m_{n,H}$ is equal to $|G_n| - O((n-1)!)$. 

Let $\pi_1, \ldots, \pi_r$ be elements of $G_n$ such that no conjugate of $\pi_i$ under the elements of $H$ is equal to a conjugate of $\pi_j$ under some element of $H$ for any $i\neq j$, and if $\pi$ is an element of $G_n$ having distinct $H$-conjugates, then $\pi$ is an $H$-conjugate of one of $\pi_1, \ldots, \pi_r$. Note that $r |H| = m_{n, H}$. Moreover, $L^2(G_n)$ contains $\oplus_{i=1}^r \bbC[H]\delta_{\pi_i}$ as an $H$-subrepresentation. Note that the subspace $\bbC[H]\delta_{\pi_i}$ of $L^2(G_n)$ is isomorphic to $\bbC[H]$ as representations of $H$. 
So, for any irreducible representation $\rho$ of $H$, 
$$
\dim L^2(G_n)_\rho
\geq 
r (\dim \rho)^2
=\frac {(\dim \rho)^2}{|H|} m_{n,H}.$$
So, for any irreducible representation $\rho$ of $H$, it follows that
\begin{align*}
\dim L^2(G_n)_\rho 
& = |G_n| - \sum_{\rho'\not \simeq \rho} \dim L^2(G_n)_{\rho'}\\
& \leq  |G_n| - \sum_{\rho'\not \simeq \rho} 
\frac {(\dim \rho')^2}{|H|} m_{n,H}\\
& =  |G_n| - 
\frac {|H| - (\dim \rho)^2}{|H|} m_{n,H}\\
& =  |G_n| - m_{n,H} +  
\frac {(\dim \rho)^2}{|H|} m_{n,H}\\
& =  (|G_n| - m_{n,H}) 
\left( 1 - \frac {(\dim \rho)^2}{|H|} \right) + 
\frac {(\dim \rho)^2}{|H|} |G_n|,
\end{align*}
where $\rho'$ runs over the irreducible representations of $H$ in the first two steps. 
Consequently, 
$$
\lim_{n\to \infty} 
\frac 
{\dim L^2(G_{n})_{ \rho}}
{\dim L^2(G_{n})}
= \frac {(\dim \rho)^2}{|H|}
$$
holds for any irreducible representation $\rho$ of $H$. 

The second part follows. 
\end{proof}

Let $S_{k,n}$ denote the subset of $\frakS_n$ defined by 
$$S_{k,n}
= T_{k, n} 
\bigcup 
\left(
\cup_{h\in H_{k,n}}
\{
h(1, 2, 3, 4, \ldots, n-1, n) h^\mo
\}
\right)
.$$
Let $\iota_n: \frakS_n \to \Aut(\frakS_n)$ denote the map which sends an element $g$ of $\frakS_n$ to the inner automorphism $x\mapsto gxg^\mo$. 

\begin{proposition}
\label{Prop:2PowerKSymm}
For any $0 < \varepsilon < 1 - \frac 1 {2^k}$, there exists an integer $n_\varepsilon$ such that for each $n \geq n_\varepsilon$, 
the spectrum of $C(\frakS_n, S_{k,n})$ contains a subset of size 
$$\geq \left(1 - \frac 1{2^k} - \varepsilon\right) n!,$$
which is symmetric about the origin, 
or $C(\frakS_n, S_{k, n})$ is non-isomorphic to the twisted Cayley graph $C(\frakS_n, S_{k,n})^{\sigma}$ for some $\sigma \in \iota_n(H_{k, n})$.
\end{proposition}

\begin{proof}
From the proof of Lemma \ref{Lemma:UniformOnSn}, it follows that 
$$
|G_n| - \dim L^2(G_n)_\rho 
\geq 
\frac {|H| - (\dim \rho)^2}{|H|} m_{n,H}
$$
with notations as in Lemma \ref{Lemma:UniformOnSn}. 
So, if $\rho$ denotes the trivial representation of $H$, then 
$$
|G_n| - \dim L^2(G_n)_\rho 
\geq 
\frac {|H| - 1}{|H|} m_{n,H}
= 
\frac {|H| - 1}{|H|} |G_n| 
- \frac {|H| - 1}{|H|}(|G_n| - m_{n,H}).
$$
Applying Theorem \ref{Thm:75Symm}, the result follows. 
\end{proof}

It is a difficult problem to determine whether two twisted Cayley graphs are conjugates. More generally, given a graph $(V, E)$, one can consider its twist by an involution of its vertex set (see \S \ref{SubSec:GabberGalil}, for instance) and consider the problem of distinguishing $(V, E)$ from its twists. In this article, this has been settled for certain twists of the Cayley graphs of symmetric groups and alternating groups, while constructing examples of av-isospectral non-isomorphic expanders (see the proof of Theorem \ref{Thm:AsymptoticallyIsoSpecNonIsom}).

\section{Ramanujan graphs from Cayley sum graphs, twisted Cayley graphs and twisted Cayley sum graphs}
\label{Sec:Ramanujan}

Relying on known examples of Ramanujan graphs, we show that there are Ramanujan graphs formed by each of the three variants of Cayley graphs. 

\begin{theorem}
\label{Thm:RamaLPSVari}
Let $p, q$ be distinct prime congruent to $1$ modulo $4$, and $i$ be an integer satisfying $i^2 \equiv -1 \pmod q$. Let $G_q$ denote the group $\mathrm{PSL}_2(\bbF_q)$ or $\mathrm{PGL}_2(\bbF_q)$ according as $p$ is a square or a non-square modulo $q$. Let $S^{p,q}$ denote the set of elements of $\mathrm{PGL}_2(\bbF_q)$ of the form 
$$
\begin{pmatrix}
a_0 + ia_1 & a_2 + ia_3\\
-a_2 + ia_3 & a_0 - ia_1
\end{pmatrix},
$$
where $a_0, a_1, a_2, a_3$ are integers with $a_0^2 + a_1 ^2 + a_2^2 + a_3^2 = p$ and $a_0-1, a_1, a_2, a_3$ are even. Let $\sigma_q$ denote the automorphism of $G_q$ induced by the conjugation action of any one of the $2\times 2$ matrices
$$
\begin{pmatrix}
0 &1\\
1 & 0 
\end{pmatrix}, 
\begin{pmatrix}
0 &1\\
-1 & 0 
\end{pmatrix}, 
\begin{pmatrix}
-1 & 0 \\
0 &1
\end{pmatrix}, 
\begin{pmatrix}
0 & 1\\
i & 0\\
\end{pmatrix}
,
\begin{pmatrix}
0 & -1\\
i & 0\\
\end{pmatrix}
,
$$
in $\gln_2(\bbF_q)$.
Then the twisted Cayley graph 
$$
C(G_q, S^{p,q})^{\sigma_q}
$$
is a Ramanujan graph. 
\end{theorem}

\begin{proof}
Note that $S^{p,q}$ is symmetric and $\sigma_q(S^{p,q}) = S^{p,q}$. By combining \cite[Theorem 4.1]{LPS88Ramanujan} with Theorem \ref{Thm:EquivExpansion}, the result follows. 
\end{proof}

Among the five Ramanujan graphs as in Theorem \ref{Thm:RamaLPSVari}, at least two of them are non-isomorphic to the Ramanujan graphs constructed by Lubotzky--Phillips--Sarnak unless around $75\%$ of the spectra of the graphs constructed by them are symmetric about the origin. More precisely, we have the following result.

\begin{theorem}
\label{Thm:75Percent}
Let $p$ be a prime congruent to $1$ modulo $4$. Then for any $0 < \varepsilon < \frac 34$, there exists an integer $q_\varepsilon$ such that for each prime $q\geq q_\varepsilon$ with $q\equiv 1 \pmod 4$, 
the spectrum of the Cayley graph $C(G_q, S^{p,q})$ contains a subset of size 
$$\geq 
\left(\frac 34 - \varepsilon\right) 
|G_q|$$
which is symmetric about the origin, or 
for some 
$$\tilde \sigma_q
\in 
\left\{
\begin{pmatrix}
0 &1\\
1 & 0 
\end{pmatrix}, 
\begin{pmatrix}
0 &1\\
-1 & 0 
\end{pmatrix}, 
\begin{pmatrix}
-1 & 0 \\
0 &1
\end{pmatrix}
\right\},
$$
$$
\tilde \tau_q
\in 
\left\{
\begin{pmatrix}
0 & 1\\
i & 0\\
\end{pmatrix}
,
\begin{pmatrix}
0 & -1\\
i & 0\\
\end{pmatrix}
,
\begin{pmatrix}
-1 & 0 \\
0 &1
\end{pmatrix}
\right\},
$$
the Cayley graph $C(G_q, S^{p, q})$ is non-isomorphic to each of the twisted Cayley graphs $C(G_q, S^{p, q})^{\sigma_q}$, $C(G_q, S^{p, q})^{\tau_q}$ 
where $\sigma_q$ (resp. $\tau_q$) denotes the conjugation action on $G_q$ induced by $\tilde \sigma_q$ (resp. $\tilde \tau_q$). Consequently, one of the following statements hold. 
\begin{enumerate}
\item 
For every $0 < \varepsilon < \frac 34$, there exists an integer $q_\varepsilon$ such that for any prime $q\geq q_\varepsilon$ with $q\equiv 1 \pmod 4$, 
the spectrum of the Cayley graph $C(G_q, S^{p,q})$ contains a subset of size 
$$\geq 
\left(\frac 34 - \varepsilon\right) 
|G_q|$$
which is symmetric about the origin.

\item 

There exist an infinite set $\calQ$ consisting of primes $\equiv 1 \pmod 4$ and elements 
$$\tilde\sigma_q
\in 
\left\{
\begin{pmatrix}
0 &1\\
1 & 0 
\end{pmatrix}, 
\begin{pmatrix}
0 &1\\
-1 & 0 
\end{pmatrix}, 
\begin{pmatrix}
-1 & 0 \\
0 &1
\end{pmatrix}
\right\},
$$
$$
\tilde\tau_q
\in 
\left\{
\begin{pmatrix}
0 & 1\\
i & 0\\
\end{pmatrix}
,
\begin{pmatrix}
0 & -1\\
i & 0\\
\end{pmatrix}
,
\begin{pmatrix}
-1 & 0 \\
0 &1
\end{pmatrix}
\right\}
$$
such that for any $q\in \calQ$, the Cayley graph $C(G_q, S^{p, q})$ is non-isomorphic to each of the twisted Cayley graphs $C(G_q, S^{p, q})^{\sigma_q}$, $C(G_q, S^{p, q})^{\tau_q}$ where $\sigma_q$ (resp. $\tau_q$) denotes the conjugation action on $G_q$ induced by $\tilde \sigma_q$ (resp. $\tilde \tau_q$). 
\end{enumerate}
\end{theorem}

\begin{proof}
Let $\sigma_q$ denote the automorphism of $G_q$ induced by the conjugation action of any one of the $2\times 2$ matrices
$$
\begin{pmatrix}
0 &1\\
1 & 0 
\end{pmatrix}, 
\begin{pmatrix}
0 &1\\
-1 & 0 
\end{pmatrix}, 
\begin{pmatrix}
-1 & 0 \\
0 &1
\end{pmatrix}, 
\begin{pmatrix}
0 & 1\\
i & 0\\
\end{pmatrix}
,
\begin{pmatrix}
0 & -1\\
i & 0\\
\end{pmatrix}
,
$$
in $\gln_2(\bbF_q)$. The result follows since the number of elements of $G_q$ fixed under the action of $\sigma_q$ is $o(|G_q|)$.
\end{proof}

Using the fact that the eigenvalues of the Cayley graph of a group with respect to a symmetric normal subset can be expressed by its characters, and the character table of $\sln_2(\bbF_q)$, Lubotzky obtained examples of Ramanujan graphs. 

\begin{theorem}
\label{Thm:Rama3Vari}
Let $q$ be a prime power with $q\equiv 3\pmod 4$. Let $G$ denote the group $\sln_2(\bbF_q)$ and $S$ denote the union of the conjugacy classes of the elements 
$$
\begin{pmatrix}
1 & 0 \\
1 & 1
\end{pmatrix}, 
\begin{pmatrix}
1 & 0 \\
-1 & 1
\end{pmatrix}
$$
of $G$. 
Let $\sigma_q$ denote the automorphism of $G$ induced by the conjugation action of any one of the $2\times 2$ matrices
$$
\begin{pmatrix}
0 &1\\
1 & 0 
\end{pmatrix}, 
\begin{pmatrix}
0 &1\\
-1 & 0 
\end{pmatrix}
$$
in $\gln_2(\bbF_q)$.
Then the graphs 
$$C_\Sigma(G, S), 
C(G, S)^{\sigma_q}, 
C_\Sigma(G, S)^{\sigma_q}
$$
are Ramanujan graphs. 
\end{theorem}

\begin{proof}
It follows from \cite[Corollary 8.2.3(ii)]{LubotzkyDiscreteGroups} and Theorem \ref{Thm:EquivExpansion}. 
\end{proof}

For a prime power $q$ with $q\equiv 1 \pmod 4$, the \textit{Paley graph} $X(q)$ is defined as the graph having the finite field $\bbF_q$ of order $q$ as its set of vertices and two vertices are adjacent if their difference is a nonzero square in $\bbF_q$. Thus the Paley graph is the Cayley graph of $\bbF_q$ with respect to the set of nonzero squares. Since $q\equiv 1\pmod 4$, it is an undirected graph.
Define the \textit{Paley sum graph} to be the Cayley sum graph of $\bbF_q$ with respect to the set of nonzero squares, and denote it by $X_\Sigma(q)$. If $q= p^{2k}$, then $x\mapsto x^{p^k}$ defines an automorphism $\sigma$ of $\bbF_q$ of order two. Define the \textit{twisted Paley graph} (resp. the \textit{twisted Paley sum  graph}) to be the twisted Cayley graph $C(\bbF_q, (\bbF_q^*)^2)^\sigma$ (resp. the twisted Cayley sum graph $C_\Sigma(\bbF_q, (\bbF_q^*)^2)^\sigma$, and denote it by $X(q)^\sigma$ (resp. $X_\Sigma(q)^\sigma$). 

\begin{theorem}
\label{Thm:Paley3Vari}
The Paley sum graph $X_\Sigma(q)$ is a Ramanujan graph. Moreover, if $q = p^{2k}$ and $\sigma$ is as above, then the twisted Paley graph $X(q)^\sigma$, and the twisted Paley sum graph $X_\Sigma(q)^\sigma$ are also Ramanujan graphs. 
\end{theorem}

\begin{proof}
It follows from \cite[Proposition 8.3.3]{LubotzkyDiscreteGroups} and Theorem \ref{Thm:EquivExpansion}. 
\end{proof}

\section{Bounded degree non-expanders using twisted Cayley graphs}
\label{Sec:NonExpander}
The following result of Lubotzky and Weiss states that the finite quotients of a finitely generated amenable group do not form an expander family. 

\begin{theorem}
[{\cite[Theorem 3.1]{LubotzkyWeissGroupsAndExpanders}}]
\label{Thm:CayleyAmenable}
Let $G$ be an amenable group generated by a finite set $\Sigma$. Let $\pi_i: G \to G_i$ be an infinite family of finite quotients of $G$. The family $\{C(G_i, \pi_i(\Sigma \cup \Sigma^\mo)) \}$ of Cayley graphs is a non-expander family. 
\end{theorem}

We use the above result to prove that a similar statement holds for the twisted Cayley graphs. 

\begin{theorem}
\label{Thm:TwistCayleyAmenable}
Let $G$ be an amenable group generated by a finite set $\Sigma$ and $\sigma$ be an inner automorphism of $G$ of order two. Let $\pi_i: G \to G_i$ be an infinite family of finite quotients of $G$. The family $\{C(G_i, \pi_i(\Sigma \cup \sigma(\Sigma^\mo)))^\sigma\}$ of twisted Cayley graphs is a non-expander family, i.e., the spectra of the normalised adjacency operators of these graphs are not uniformly bounded away  from $1$. 
\end{theorem}

\begin{proof}
On the contrary, let us assume that the spectra of the normalised adjacency operators of the graphs $\{C(G_i, \pi_i(\Sigma \cup \sigma(\Sigma^\mo)))^\sigma\}$ are uniformly bounded away from $1$. By the discrete Cheeger--Buser inequality, the edge Cheeger constants of these graphs are uniformly bounded away from $0$, and hence so do their vertex Cheeger constants. It follows that the vertex Cheeger constants of the twisted Cayley graphs $\{C(G_i, \pi_i(\Sigma \cup \Sigma^\mo \cup \sigma(\Sigma \cup \Sigma^\mo)))^\sigma\}$ are uniformly bounded away from $0$. Note that this is a family of graphs of bounded degree. By Theorem \ref{Thm:CombiToSpecForOne}, the spectra of the normalised adjacency operators of these graphs are uniformly bounded away from $1$. By Theorem \ref{Thm:EquivExpansion}, the spectra of the normalised adjacency operators of the Cayley graphs $\{C(G_i, \pi_i(\Sigma \cup \Sigma^\mo \cup \sigma(\Sigma \cup \Sigma^\mo)))\}$ graphs are uniformly bounded away from $1$. This is impossible by Theorem \ref{Thm:CayleyAmenable}. This proves that the family $\{C(G_i, \pi_i(\Sigma \cup \sigma(\Sigma^\mo)))^\sigma\}$ of twisted Cayley graphs is a non-expander family. 
\end{proof}

\section{Bounded degree expanders using twisted Cayley graphs}
\label{Sec:Expanders}

It has been proved by Roichman that the Cayley graphs of symmetric groups (resp. alternating groups) with respect to the conjugacy class of long odd (resp. even) cycles are expanders \cite{RoichmanExpansionCayleyAlt}. There are several interesting questions about expansion in Cayley graphs of symmetric and alternating groups. For instance, it has been asked by several authors whether the Cayley graphs of symmetric groups or of alternating groups form an expander family for some choice of generating sets of bounded size, see the works of Babai--Hetyei--Kantor--Lubotzky--Seress \cite[Problem 2.2]{BabaiHetyeiKantorLubotzkySeressDiamFiniteGrp}, Lubotzky--Weiss \cite[Problem 4.2]{LubotzkyWeissGroupsAndExpanders}, Lubotzky \cite[p. 165]{LubotzkyCayleyGraphSurvey}, \cite[Problem 10.3.4]{LubotzkyDiscreteGroups2010}. This question has been answered for symmetric groups and alternating groups by Kassabov \cite{KassabovSymmetricGrpExpander}. Using these expanders on alternating groups, Rozenman, Shalev and Wigderson provided a recursive construction of expanders \cite{RozenmanShalevWigdersonIterative}.

\begin{lemma}
\label{Lemma:CayleyToTwistedCayley}
If $S$ is a symmetric subset of a group $G$ and $C(G, S)$ is an $(n, d, \varepsilon)$-expander, then for any order two automorphism of $G$, the twisted Cayley graph $C(G, S \cup \sigma(S))^\sigma$ has degree $\leq 2d$, it is an $\varepsilon'$-vertex expander with $\varepsilon'\geq \frac{\varepsilon^2} {40^4 d^2}$, and it is a two-sided $\delta$-expander with $\delta\geq \frac{\varepsilon^2} {8 \cdot 10^4 \cdot d^2}$. 
\end{lemma}

\begin{proof}
If $C(G, S)$ is an $(n, d, \varepsilon)$-expander, then for any order two automorphism $\sigma$ of $G$, the Cayley graph $C(G, S \cup \sigma(S))$ is an $(n, d', \varepsilon)$-expander with $d'\leq 2d$. So, the Cayley graph $C(G, S \cup \sigma(S))$ is a two-sided $\delta$-expander with $\delta = \frac{\varepsilon^2}{2\cdot 10^4 \cdot d'^2} \geq \frac{\varepsilon^2} {8 \cdot 10^4 \cdot d^2}$ by the discrete Cheeger--Buser inequality and \cite[Theorem 2.6]{CayleyBottomBipartite}. By Theorem \ref{Thm:EquivExpansion}, the twisted Cayley graph $C(G, S \cup \sigma(S))^\sigma$ is also a two-sided $\delta$-expander and this graph has degree $\leq 2d$. By the discrete Cheeger--Buser inequality, $C(G, S \cup \sigma(S))^\sigma$ is an $\varepsilon'$-vertex expander with $\varepsilon'\geq \frac{\varepsilon^2} {40^4 d^2}$.
\end{proof}

\subsection{Twisted Cayley graphs of permutation groups}

\begin{theorem}
\label{Thm:TwistCayleySymmExp}
For each $n$, let $\sigma_n$ (resp. $\tilde \sigma_n$) denote an automorphism of $\frakA_n$ (resp. $\frakS_n$) of order $\leq 2$. Then, there exist constants $M>0$ and $\varepsilon>0$ and there exist subsets $F_n$ of the alternating group $\frakA_n$ and $\tilde F_n$ of the symmetric group $\frakS_n$ for any $n$ such that each of the twisted Cayley graphs $C(\frakA_n, F_n)^{\sigma_n}$ and $C(\frakS_n, \tilde F_n)^{\tilde \sigma_n}$ forms a family of $\varepsilon$-expanders of degree bounded by $M$. 
\end{theorem}

\begin{proof}
By \cite[Theorem 2]{KassabovSymmetricGrpExpander}, there exist universal constants $L > 0,\epsilon > 0$ such that for every $n$, there exists a generating set $S_n$ of $\frakA_n$ of size $\leq L$ such that the Cayley graphs $C(\frakA_n, S_n)$ form a family of $\epsilon$-expanders. Moreover, there exists a generating set $\tilde S_n$ of $\frakS_n$ with the same property. Applying Lemma \ref{Lemma:CayleyToTwistedCayley}, the result follows. 
\end{proof}

From the proof of Theorem \ref{Thm:AsymptoticallyIsoSpecNonIsom} on the existence of families of av-isospectral non-isomorphic expanders of bounded degree formed by the twisted Cayley graphs of symmetric groups and by the twisted Cayley graphs of alternating groups, it follows that for appropriate inner automorphisms $\sigma_n$ and $\tilde \sigma_n$, the twisted Cayley graph $C(\frakA_n, F_n)^{\sigma_n}$ (resp. $C(\frakS_n, \tilde F_n)^{\tilde \sigma_n}$) is non-isomorphic to $C(\frakA_n, F_n)$ (resp. $C(\frakS_n, \tilde F_n)$).

\subsection{Twisted Cayley graphs of special linear groups}

Further examples of bounded degree expanders can be constructed using finite simple groups of Lie type as base groups. As mentioned previously in the introduction, the first explicit constructions of expanders were those by Margulis \cite{MargulisExpanders}. These are made of the $4$-regular Cayley graphs $\Gamma_p=C(\sln_{2}(\mathbb F_p), \{A_p,B_p, A_p^\mo, B_p^\mo\})$ with the matrices $A_p,B_p$ being images of the Sanov's generators. More generally, Bourgain--Varj\'u in \cite{BourgainVarjuExpansionSLdZqZ}, showed that if $S$ is a finite, symmetric subset of $G = \sln_{d}(\mathbb{Z})$, such that $S$ generates a Zariski-dense subgroup of $G$, then the Cayley graphs $C(G_{q}, S_{q})$ forms an expander family, where $q$ runs through the integers. Combining the above results with the technique of taking twists, we obtain the following bounded degree expanders. 

\begin{theorem}
\label{Thm:SL2Expander}
Let $\ell$ be a positive integer, and $A, B\in \sln_2(\bbZ)$ denote the matrices defined by 
$$
A 
= 
\begin{pmatrix}
1 & a\\ 
0 & 1
\end{pmatrix}
, 
\quad 
B = 
\begin{pmatrix}
1 & 0 \\
b & 1
\end{pmatrix}
$$
where $a, b \geq 2$ are integers. Let $\sigma_p$ denote the automorphism of $\sln_2(\bbFp)$ induced by the conjugation action of the $2\times 2$ matrix 
$$
\begin{pmatrix}
0 & 1\\
1 & 0 
\end{pmatrix}
$$
of order two in $\gln_2(\bbF_p)$, and $S_p$ denote the subset 
$$
\{A_p, B_p, A_p^\mo, B_p^\mo\}
\cup 
\sigma_p(\{A_p, B_p, A_p^\mo, B_p^\mo\})
$$
of $\sln_2(\bbFp)$ of size $\leq 8$. 
Then 
$$
\{
C(\sln_2(\bbF_p), S_p)^{\sigma_p}
\}_{p\geq \max\{a, b\}}
$$
forms a family of expanders. 
\end{theorem}

\begin{proof}
By combining \cite{MargulisExpanders} with Lemma \ref{Lemma:CayleyToTwistedCayley}, the result follows.
\end{proof}

In the context of special linear groups, Bourgain--Varj\'u showed that if $S$ is a finite, symmetric subset of $G = \sln_{d}(\mathbb{Z})$, such that $S$ generates a Zariski-dense subgroup of $G$, then the Cayley graphs $C(G_{q}, S_{q})$ forms an expander family, where $q$ runs through the integers \cite{BourgainVarjuExpansionSLdZqZ}. As a consequence of their result, one deduces the following.

\begin{theorem}
\label{Thm:BVQuotients}
For each integer $q\geq 1$, let $\sigma_q$ be an automorphism of $\sln_{d}(\mathbb{Z}/q\bbZ)$ of order $\leq 2$. Let $S$ be a finite, symmetric subset of $\sln_{d}(\mathbb{Z})$ which generates a Zariski-dense subgroup of $\sln_{d}(\mathbb{Z})$. Then the twisted Cayley graphs $\{C(\sln_d(\mathbb{Z}/q\mathbb{Z}), S_q \cup \sigma_q(S_q))^{\sigma_q}\}$ form an expander family where $q$ runs over the positive integers.
\end{theorem}

\begin{proof}
By combining \cite{BourgainVarjuExpansionSLdZqZ} with Lemma \ref{Lemma:CayleyToTwistedCayley}, the result follows.
\end{proof}

For examples of certain explicit symmetric subsets of $\sln_d(\bbZ)$ which generate a Zariski-dense subgroup of $\sln_d(\bbZ)$, we refer to a recent work of Arzhantseva and the first author \cite{ArzhantsevaBiswasDgBdd}. 

The expansion properties of subsets of special linear groups have also been studied by Helfgott \cite{HelfgottGrowthGenerationSL2Fp}, Bourgain--Gamburd \cite{BourgainGamburdUniformExpanBdd, BourgainGamburdExpansionRandomWalkSLdI, BourgainGamburdExpansionRandomWalkSLdII}, Bourgain--Gamburd--Sarnak \cite{BourgainGamburdSarnakAffineLinearSieveExpanderSumProd}, Breuillard--Gamburd \cite{BreuillardGamburdStrongUnifExpans}, Varj\'{u} \cite{VarjuExpansionSLdModISqFree}, Kowalski \cite{KowalskiExplicitGrowthExpansionSL2}, Bradford \cite{BradfordExpansionRandomWalkSieving}. They have constructed expander Cayley graphs, which  can also be twisted by involutions to obtain examples of expanders.

\subsection{Twisted Cayley graphs of finite simple groups of Lie type}

From a series of works by Helfgott \cite{HelfgottGrowthGenerationSL2Fp}, 
Kassabov--Lubotzky--Nikolov \cite{KassabovLubotzkyNikolovFiniteSimpleGrpExp}, 
Breuillard--Green--Tao \cite{BreuillardGreenTaoApproxSubgrpLinGrp}, 
Breuillard--Green--Tao \cite{BGTSuzuki}, 
Golsefidy--Varj\'{u} \cite{GolsefidyVarjuExpansionPerfectGrp}, 
Breuillard--Green--Guralnick--Tao \cite{BGGTExpansionSimpleLie}, 
Pyber--Szab\'{o} \cite{PyberSzaboGrowthFinSimpleGrpLie}, it is known that the Cayley graphs associated to random pairs of elements in finite simple groups of Lie type are expanders. This yields plenty of examples of vertex transitive expanders. We refer to the survey articles by Helfgott \cite{HelfgottGrowthInGroupsIdeaPersp, HelfgottGrowthExpansionAlgGrpFinField}. 

We show that an analogue of \cite[Theorem 1.2]{BGGTExpansionSimpleLie} holds for twisted Cayley graphs afforded by finite simple groups of Lie type with respect to random set of generators. In particular, this proves that there are plenty of examples of non-vertex-transitive regular expanders of group-theoretic origin. 

\begin{theorem}
\label{Thm:FSGLTwisted}
Suppose $G$ is a finite simple group of Lie type and $\sigma$ is an automorphism of $G$ of order two. Let $a, b$ be two elements of $G$ selected uniformly at random and $S$ denote the set $\{a, b\}$. Let $T$ denote the multiset $S \sqcup S^\mo \sqcup \sigma(S) \sqcup \sigma(S^\mo)$. Then the twisted Cayley graph $C(G, T)^\sigma$ is $\varepsilon$-expanding with probability at least $1 - C|G|^{-\delta}$ where $C, \varepsilon, \delta>0$ are constants depending only on the rank of $G$. 
\end{theorem}

\begin{proof}
By \cite[Theorem 1.2]{BGGTExpansionSimpleLie}, there exist constants $C, \epsilon, \delta>0$ depending only on the rank of $G$ such that the Cayley graph $C(G, S\sqcup S^\mo)$ is a two-sided $\epsilon$-expander with probability $\geq p =  1 - C |G|^{-\delta}$. By the discrete Cheeger--Buser inequality, its vertex Cheeger constant is $\geq \frac \epsilon 2$ with probability $\geq p$. So, the vertex Cheeger constant of the Cayley graph $C(G, T)$ is $\geq \frac \epsilon 2$ with probability $\geq p$. By the discrete Cheeger--Buser inequality and \cite[Theorem 2.6]{CayleyBottomBipartite}, it follows that the Cayley graph $C(G, T)$ is a two-sided $\eta$-expander with $\eta \geq \frac {\epsilon^2} {2 \cdot 10^4 \cdot 8^2 \cdot 4}$ with probability $\geq p$. Taking $\varepsilon = \frac{\epsilon^2}{2^9 \cdot 10^4}$, the result follows from Theorem \ref{Thm:EquivExpansion}.
\end{proof}

\subsection{Twists of Gabber--Galil expanders}
\label{SubSec:GabberGalil}

In the above, the technique of `twisting by involutions' has been applied to Cayley graphs with certain properties to obtain examples of graphs with similar properties. In the following, we show that the technique of `twisting by involutions' can also be applied to Schreier graphs having certain spectral properties to obtain examples of graphs with similar properties. 

Gabber and Galil provided an explicit example of expander families  \cite{GabberGalilLinearSized}. In this section, we show that the key strategy that underlies the proofs of \S \ref{Sec:Ramanujan}, \ref{Sec:Expanders}, as outlined in the paragraph preceding Lemma \ref{Lemma:EigenvalueXY}, can be applied to the Gabber--Galil expanders to obtain further examples of explicit expanders. 

Consider the four maps $S, T, U, V$ from $\bbZ^2$ to $\bbZ^2$, defined by 
\begin{align*}
S(x) 
& 
= 
\begin{pmatrix}
1 & 0 \\
1 & 1
\end{pmatrix}
x,\\
T(x) 
& 
= 
\begin{pmatrix}
1 & 1 \\
0 & 1
\end{pmatrix}
x,\\
U(x) 
& 
= 
x + 
\begin{pmatrix}
1\\
0
\end{pmatrix},\\
V(x) 
& 
= 
x + 
\begin{pmatrix}
0\\
1
\end{pmatrix}
\end{align*}
for any $x\in \bbZ^2$. 

Let $G_n$ denote the graph having $\bbZ^2/n\bbZ^2$ as its set of vertices, and two vertices $x, y$ are adjacent if $y = \varphi (x)$ for some $\varphi \in \{S^\pm, T^\pm, U^\pm, V^\pm\}$. Gabber and Galil proved that the nontrivial spectrum of the adjacency operator of $G_n$ lies in the interval $[-2(1 + 2\sqrt 2), 2(1+ 2\sqrt 2)]$.

Let $\sigma: \bbZ^2/n\bbZ^2\to \bbZ^2/n\bbZ^2$ be a map of order two. Let $G_n^\sigma$ denote the graph having $\bbZ^2/n\bbZ^2$ as its set of vertices, and two vertices $x, y$ are adjacent if $y = \sigma(\varphi (x))$ for some $\varphi \in \{S^\pm, T^\pm, U^\pm, V^\pm\}$. Note that this graph is undirected if 
$$\sigma \{S^\pm, T^\pm, U^\pm, V^\pm\}
\sigma = \{S^\pm, T^\pm, U^\pm, V^\pm\}$$
holds. Indeed, if the above equality holds, and $y = \sigma(\varphi(x))$ holds for two vertices with $\varphi \in \{S^\pm, T^\pm, U^\pm, V^\pm\}$, then $x = \sigma(\psi(y))$ where $\psi$ is the unique element of $\{S^\pm, T^\pm, U^\pm, V^\pm\}$ satisfying $\psi = \sigma \varphi^\mo \sigma$.

For $\varepsilon = \pm 1$, let $\tau_\varepsilon: \bbZ^2\to \bbZ^2$ denote the map which sends $(a, b)$ to $\varepsilon (b, a)$. 
For $(\alpha, \beta) \in \{(1, 1), (1, -1), (-1, 1), (-1, -1)\}$, let $\tau_{(\alpha, \beta)}: \bbZ^2\to \bbZ^2$ denote the map which sends $(a, b)$ to $(\alpha a, \beta b)$. Note that the five maps $\tau_1, \tau_{-1}, \tau_{(1, -1)}, \tau_{(-1, 1)}, \tau_{(-1, -1)}$ on $\bbZ^2$ are of order $2$, and the actions of these on $\bbZ^2$ descend to $\bbZ^2/n\bbZ^2$. 
We claim that the graph $G_n^\theta$ is undirected for any automorphism $\theta$ among $\tau_1, \tau_{-1}, \tau_{(1, -1)}, \tau_{(-1, 1)}, \tau_{(-1, -1)}$. For $\varepsilon = \pm 1$, it follows that $\tau_\varepsilon S \tau_\varepsilon = T, \tau_\varepsilon U \tau_\varepsilon = V^\varepsilon$, which implies that 
$$\tau_\varepsilon \{S^\pm, T^\pm, U^\pm, V^\pm\}
\tau_\varepsilon = \{S^\pm, T^\pm, U^\pm, V^\pm\}.$$
For $(\alpha, \beta) \in \{(1, 1), (1, -1), (-1, 1), (-1, -1)\}$,  it follows that 
\begin{align*}
\tau_{(\alpha, \beta)} S \tau_{(\alpha, \beta)}
& = S^{\alpha \beta}, \\
 \tau_{(\alpha, \beta)} T \tau_{(\alpha, \beta)} 
& = T^{\alpha\beta}, \\
\tau_{(\alpha, \beta)} U \tau_{(\alpha, \beta)} 
& = U^{\alpha}, \\
\tau_{(\alpha, \beta)} V \tau_{(\alpha, \beta)} 
& = V^{\beta},
\end{align*}
which implies that 
$$\tau_{(\alpha, \beta)} \{S^\pm, T^\pm, U^\pm, V^\pm\}
\tau_{(\alpha, \beta)} = \{S^\pm, T^\pm, U^\pm, V^\pm\}.$$
This proves the claim that the graph $G_n^\theta$ is undirected for any automorphism $\theta$ among $\tau_1, \tau_{-1}, \tau_{(1, -1)}, \tau_{(-1, 1)}, \tau_{(-1, -1)}$. 
Note that the group 
\begin{equation}
\label{Eqn:Klein1}
\{\tau_{(1, 1)}, \tau_{(1, -1)}, \tau_{(-1, 1)}, \tau_{(-1, -1)}\}
\end{equation}
is isomorphic to the Klein $4$-group. 
Since 
$\tau_1 \tau_{-1} 
= \tau_{-1} \tau_1
= \tau_{(-1, -1)}
$
holds, it follows that the group 
\begin{equation}
\label{Eqn:Klein2}
\{\tau_{(1, 1)}, \tau_{1}, \tau_{-1}, \tau_{(-1, -1)}\}
\end{equation}
is also isomorphic to the Klein $4$-group. 
Since 
$\tau_\varepsilon \tau_{(\alpha, \beta)} \tau_\varepsilon
= 
\tau_{(\beta, \alpha)}
$
holds, it follows that 
\begin{equation}
\label{Eqn:Klein3}
\{\tau_{(1, 1)}, \tau_{\varepsilon}, \tau_{(\alpha, \alpha)},
\tau_{\varepsilon} \tau_{(\alpha, \alpha)}
\}, 
\quad 
\varepsilon = \pm 1, 
\alpha = \pm 1
\end{equation}
is also isomorphic to the Klein $4$-group.

\begin{theorem}
\quad 
\begin{enumerate}
\item 
Let $\theta$ be one among the five automorphisms
$\tau_1, \tau_{-1}, \tau_{(1, -1)}, \tau_{(-1, 1)}, \tau_{(-1, -1)}$. 
The nontrivial spectrum of the adjacency operator of $G_n^\theta$ lies in the interval $[-2(1 + 2\sqrt 2), 2(1+ 2\sqrt 2)]$. Moreover, the following 
$$
\{G_n^{\tau_1}\}_{n\geq 1}, 
\{G_n^{\tau_{-1}}\}_{n\geq 1}, 
\{G_n^{\tau_{(1, -1)}}\}_{n\geq 1}, 
\{G_n^{\tau_{(-1, 1)}}\}_{n\geq 1}, 
\{G_n^{\tau_{(-1, -1)}}\}_{n\geq 1}
$$
form five families of expanders of bounded degree and their spectra coincide after taking absolute values.

\item 
Let $H$ denote a finite subgroup of the group generated by 
$\tau_1, \tau_{-1}, \tau_{(1, -1)}, \tau_{(-1, 1)}, \tau_{(-1, -1)}$. 
Then 
$$
\{\{G_n^h\}_{n\geq 1}\}_{h\in H} 
$$
is an $H$-uniform expander families. 
In particular, if $H$ denotes one among the groups 
as in Equations \eqref{Eqn:Klein1}, \eqref{Eqn:Klein2}, \eqref{Eqn:Klein3}, 
then 
$$
\{\{G_n^h\}_{n\geq 1}\}_{h\in H} 
$$
is a uniformly $(\bbZ/2\bbZ)^2$-isospectral expander families. 
\item 

For any $0 < \varepsilon < \frac 34$, there exists an integer $n_\varepsilon$ such that for each $n \geq n_\varepsilon$, 
the spectrum of $G_n$ contains a subset of size 
$$\geq \left(\frac 34- \varepsilon\right) n^2,$$
which is symmetric about the origin, 
or 
for each $H$ among the groups 
as in Equations \eqref{Eqn:Klein1}, \eqref{Eqn:Klein2}, \eqref{Eqn:Klein3}, there exists an element $\theta_H\in H$ such that $G_n$ is non-isomorphic to the twisted Cayley graphs $G_n^{\theta_H}$. 
\end{enumerate}
\end{theorem}

\begin{proof}
Let $A$ (resp. $A_\theta$) denote the adjacency operator of $G_n$ (resp. $G_n^\theta$). Note that $\theta$ is of order two, and the graphs $G_n, G_n^\theta$ are undirected. Also note that the neighbourhood of a vertex $v$ of $G_n^\theta$ is equal to the neighbourhood of $\theta(v)$ in $G_n$. It follows from the discussion in the paragraph preceding Lemma \ref{Lemma:EigenvalueXY} that up to factors of $\pm 1$, the eigenvalues of $A$ are equal to the eigenvalues of $A_\theta$. The nontrivial spectrum of the adjacency operator of $G_n^\theta$ lies in the interval $[-2(1 + 2\sqrt 2), 2(1+ 2\sqrt 2)]$.

Let $P$ denote the permutation operator $\sum_{v\in \bbZ^2/n\bbZ^2} \delta_{v, \theta (v)}$, which defines an element of $L^2(\bbZ^2/n\bbZ^2)$. Since $P, A$ commute and $A_\theta$ is equal to the product of $P$ and the adjacency operator $A$, it follows that the spectrum of any one of the graphs 
$$
G_n^{\tau_1},
G_n^{\tau_{-1}},
G_n^{\tau_{(1, -1)}},
G_n^{\tau_{(-1, 1)}},
G_n^{\tau_{(-1, -1)}}
$$
coincides with the spectrum of $G_n$ after taking absolute values. 
This proves part (1). 

Let $H$ be a subgroup of $SA_2(\bbZ)$. 
For any element $\theta\in H$, there are at most $n$ elements of $(\bbZ/n\bbZ)^2$ which are left fixed by $\theta$. Hence, there are at least $n^2-(|H|-1)n$ elements in $(\bbZ/n\bbZ)^2$ such that each of them have distinct images under the action of the elements of $H$. 
Following the argument of Lemma \ref{Lemma:UniformOnSn}, it follows that the action of $H$ on $L^2((\bbZ/n\bbZ)^2)$ is uniform as $n\to \infty$. 
So, 
$
\{\{G_n^h\}_{n\geq 1}\}_{h\in H} 
$
is an $H$-uniform expander families. 
By Proposition \ref{Prop:IsospecExample}(3), 
$$
\{\{G_n^h\}_{n\geq 1}\}_{h\in H} 
$$
is a uniformly $(\bbZ/2\bbZ)^2$-isospectral expander families if $H$ denotes one among the groups as in Equations \eqref{Eqn:Klein1}, \eqref{Eqn:Klein2}, \eqref{Eqn:Klein3}. Then part (3) follows from the proof of Proposition \ref{Prop:2PowerKSymm}. 
\end{proof}

\section{Uniformly $H$-isospectral $s$-non-isomorphic expander families}
\label{Sec:AsympIsoSpect}
\subsection{Twisted Cayley graphs}
Relying on a result of Kassabov on the existence of bounded degree expanders on symmetric groups and alternating groups \cite{KassabovSymmetricGrpExpander}, we show that Theorem \ref{Thm:EquivExpansion} can be applied to twisted Cayley graphs to obtain uniformly $H$-isospectral $s$-non-isomorphic expander families of bounded degree with $H = (\bbZ/2\bbZ)^k$ and $s = \log |H|$ (Theorem \ref{Thm:AsymptoticallyIsoSpecNonIsom}). Let us provide an outline of the strategy. 

Theorem \ref{Thm:EquivExpansion} states that the Cayley graph of a group and its twisted variants (for instance, the twisted Cayley graph with respect to some order two automorphism) have their eigenvalues related by factors of $\pm 1$ (under appropriate conditions). So, we could consider the twisted Cayley graphs with respect to several order two automorphisms of the underlying group $G$. It turns out that they are $H$-isospectral under suitable conditions, where $H$ is a $2$-torsion group. Using this strategy and expanders on symmetric groups, alternating groups, special linear groups (due to Kassabov \cite{KassabovSymmetricGrpExpander}, Bourgain--Varj\'u \cite{BourgainVarjuExpansionSLdZqZ} among others), we prove that for any integer $k\geq 2$, there are $k$ families of graphs, afforded by each of these class of groups, which are expanders of bounded degree and $H$-isospectral for some $2$-torsion group $H$. Further, using bounded degree expanders on the symmetric groups and alternating groups due to Kassabov, we show that for any integer $k\geq 2$, there are $k$ families of graphs, afforded by each of these class of groups, which are expanders of bounded degree, non-isomorphic and $H$-isospectral for some $2$-torsion group $H$. Our construction of such families involves twisted Cayley graphs. 

Before proceeding further, we first provide a sufficient condition for two twisted Cayley graphs to be isomorphic. 

\subsection{$H$-isospectral expander families}
The following result implies that there are families of $H$-isospectral expander families (see Corollaries \ref{Cor:AsymptoticallyIsoSpecPerm}, \ref{Cor:AsymptoticallyIsoSpecSLnPSLd}). 

Let $\sigma: X \to X$ be a bijection. Let $P_\sigma$ denote the linear automorphism of $L^2(X)$ which sends $\delta_x$ to $\delta_{\sigma(x)}$. 
Note that 
$$\sigma\mapsto P_\sigma$$
defines an injective group homomorphism 
$$\Aut(X)
\hookrightarrow \gln (L^2(X)).$$

\begin{proposition}
\label{Prop:IsospecExample}
Let $C(G, S)$ be an undirected Cayley graph. 
Let $\calI$ be a nonempty set consisting of group automorphisms of $G$ of order two. Let $H$ denote the subgroup of the group of automorphisms of $G$ generated by $\calI$. Let $HS$ denote the subset $\cup_{h\in H} h(S)$ of $G$. 
\begin{enumerate}
\item
For any $h\in H$, the eigenvalues of $C(G, HS)^h$ coincide with the eigenvalues of $C(G, HS)$ up to certain factors of roots of unity. 
\item 
Let $h$ be an element of $H$. 
If the size of the set $\{h^2(g)g^\mo|g\in G\}$ is larger than the size of $HS$, then $C(G, HS)^h$ is a directed graph, and its spectrum coincides with that of $C(G, HS)$ after taking absolute values. 
\item 
If $H$ is abelian, then the collection $\{C(G, HS)^{h}\}_{h\in H}$ of graphs is $H$-isospectral. 
\end{enumerate}
\end{proposition}

\begin{proof}
Note that for any $h\in \calI$, $h(HS) = (HS)^\mo$ holds and hence the graph $C(G, HS)^h$ is undirected. So, the adjacency operator $A$ of $C(G, S)$ commutes with $P_\sigma$ for any $\sigma\in \calI$. It follows that $A$ of $C(G, S)$ commutes with the action of $H$ on $L^2(G)$. Note that $L^2(G)$ decomposes into the $\rho$-isotypic components of $H$ as $\rho$ runs over the irreducible representations of $H$. Denote the $\rho$-isotypic component of $L^2(G)$ by $L^2(G)_\rho$. Note that such an isotypic component is stable under the action of $A$. Denote the multi-set of eigenvalues of the restriction of $A$ to $L^2(G)_\rho$ by $\calE(A)_\rho$. For any $h\in H$, the $L^2(G)_\rho$ is stable under the action of $P_hA$, and its eigenvalues are obtained by multiplying the eigenvalues of $A$ with the eigenvalues of $P_h$ in some order. So, the adjacency operator of $C(G, HS)^h$ is equal to $P_hA$. Hence, for any $h\in H$, the eigenvalues of $C(G, HS)^h$ coincide with the eigenvalues of $C(G, HS)$ up to certain factors of roots of unity. This proves the first part. 

Let $\tau$ be an automorphism of $G$. Note that the twisted Cayley graph $C(G, S)^\tau$ is undirected if for any $x\in G, s\in \calS$, $y\in G$ with $y = \tau(xs)$, there exists $t\in S$ such that $x = \tau(yt)$, i.e., $t = \tau(s^\mo) \tau^\mo(\tau^2(x^\mo) x)$ lies in $S$. It follows that if for any element $h\in H$, the size of the set $\{h^2(g)g^\mo|g\in G\}$ is larger than the size of $HS$, then $C(G, HS)^h$ is a directed graph. The first part implies that the spectrum of $C(G, HS)^h$ coincides with that of $C(G, HS)$ after taking absolute values. 

Note that if $H$ is abelian, then $\calE(P_hA)_\chi$ is equal to $\chi(h)\calE(A)_\chi$ for any character $\chi$ of $H$. The third part follows. 
\end{proof}

\begin{theorem}
\label{Thm:TwistCayleyExpMulti}
Let $\{C(G_n, S_n)\}_{n\geq 1}$ be a family of expanders having degree bounded by $d$. Let $k\geq 2$ be an integer. For each $n$, let $\sigma_{1n}, \sigma_{2n}, \ldots, \sigma_{kn}$ be pairwise  commuting automorphisms of $G_n$ of order $\leq 2$. Let $T_n$ denote the subset of $G_n$ defined by 
$$T_n
= S_n 
\cup 
\left(\cup_{i} \sigma_{in}(S_n) \right)
\cup 
\left(\cup_{i, j} (\sigma_{in} \circ \sigma_{jn}) (S_n) \right)
\cup 
\cdots 
\cup 
(\sigma_{1n} \circ \sigma_{2n} \circ \cdots \circ \sigma_{kn})(S_n).$$
Then the following statements hold. 
\begin{enumerate}
\item 
Each of the twisted Cayley graphs 
$C(G_n, T_n)^{\sigma_{1n}}$, 
$C(G_n, T_n)^{\sigma_{2n}}, \ldots, 
C(G_n, T_n)^{\sigma_{kn}}$
forms a family of expanders of degree bounded by $2^k d$.

\item The eigenvalues of any of two of the twisted Cayley graphs 
$C(G_n, T_n)^{\sigma_{1n}}$, 
$C(G_n, T_n)^{\sigma_{2n}}, \ldots$, 
$C(G_n, T_n)^{\sigma_{kn}}$ 
are related by factors of $\pm 1$. 
\end{enumerate}
\end{theorem}

\begin{proof}
Note that if $C(G, S)$ is an $(n, d, \varepsilon)$-expander, then for pairwise commuting order two automorphisms $\sigma_1, \ldots, \sigma_k$ of $G$, the Cayley graph $C(G, T)$ is an $(n, d', \varepsilon)$-expander with $d'\leq 2^k d$, where 
$$T 
= S 
\cup 
\left(\cup_{i} \sigma_i(S) \right)
\cup 
\left(\cup_{i, j} (\sigma_i \circ \sigma_j) (S) \right)
\cup 
\cdots 
\cup 
(\sigma_1 \circ \sigma_2 \circ \cdots \circ \sigma_k)(S).$$
So, the Cayley graph $C(G, T)$ is a two-sided $\delta$-expander with $\delta = \frac{\varepsilon^2}{2\cdot 10^4 \cdot d'^2} \geq \frac{\varepsilon^2} {2^{2k+1} \cdot 10^4 \cdot d^2}$ by the discrete Cheeger--Buser inequality and \cite[Theorem 2.6]{CayleyBottomBipartite}. By Theorem \ref{Thm:EquivExpansion}, for any $1\leq i\leq k$, the twisted Cayley graph $C(G, T)^{\sigma_i}$ is also a two-sided $\delta$-expander and this graph has degree $\leq 2^k d$. The first part follows. The second part follows from Theorem \ref{Thm:EquivExpansion}. 
\end{proof}

If a sequence of groups afford an expander family of bounded degree, and for a given $k\geq 1$, if each of these groups admit $k$ automorphisms of order $\leq 2$, then Theorem \ref{Thm:TwistCayleyExpMulti} yields $k$ av-isospectral expander families. This is illustrated below in the context of symmetric groups and alternating groups. 

\begin{corollary}
\label{Cor:AsymptoticallyIsoSpecPerm}
For any $2$-torsion group $H$, there exist $H$-isospectral expander families of bounded degree formed by the twisted Cayley graphs of symmetric groups, and by the twisted Cayley graphs of alternating groups. 
\end{corollary}

\begin{proof}
For each large enough $n$, let $\tau_{1n}, \ldots, \tau_{kn}$ be order two elements of $\frakS_n$ such that any two of them commute and these elements generate a subgroup of order $2^k$. Let $\sigma_{1n}, \sigma_{2n}, \ldots, \sigma_{kn}$ (resp. $\tilde \sigma_{1n}, \tilde \sigma_{2n}, \ldots, \tilde \sigma_{kn}$) denote the automorphisms of $\frakA_n$ (resp. $\frakS_n$) induced by the conjugation action of $\tau_{1n}, \ldots, \tau_{kn}$. The corollary follows from \cite[Theorem 2]{KassabovSymmetricGrpExpander}, Theorem \ref{Thm:TwistCayleyExpMulti}, Proposition \ref{Prop:IsospecExample}. 
\end{proof}

\begin{corollary}
\label{Cor:AsymptoticallyIsoSpecSLnPSLd}
For any $2$-torsion group $H$, there exist $H$-isospectral expander families of bounded degree formed by twisted Cayley graphs of the group of $\bbFp$-points of the special linear groups (where $p$ runs over the primes), and by the $\bbF_q$-points of $\psl_d$ for any $d\geq 5$ with $d\neq 6$ (where $q$ runs over primes powers). 
\end{corollary}

\begin{proof}
For a permutation $\tau$ of $\{1, 2, \ldots, m\}$, let $P_\tau$ denote the permutation matrix, i.e., the matrix $\sum_{i=1}^m \delta_{i, \tau^\mo(i)}$ where $\delta_{i,j}$ denotes the $m\times m$ matrix having all entries equal to $0$ except one, which is equal to $1$ and lies at the $(i,j)$-th entry. 

Let $n\geq 2k$ be an integer. Let $\tau_{1}, \ldots, \tau_{k}$ be order two elements of $\frakS_n$ such that any two of them commute and these elements generate a subgroup of order $2^k$. Let $\sigma_{1p}, \sigma_{2p}, \ldots, \sigma_{kp}$ denote the automorphisms of $\sln_n(\bbFp)$ induced by the conjugation action of $P_{\tau_{1}}, \ldots, P_{\tau_{k}} \pmod p$. The corollary follows from \cite{MargulisExpanders}, \cite{BourgainVarjuExpansionSLdZqZ}, \cite{ArzhantsevaBiswasDgBdd}, Theorem \ref{Thm:TwistCayleyExpMulti}, Proposition \ref{Prop:IsospecExample}. 

Let $\sigma'_{1q}, \sigma'_{2q}, \ldots, \sigma'_{kq}$ denote the automorphisms of $\psl_d(\bbF_q)$ induced by the conjugation action of $P_{\tau_{1}}, \ldots, P_{\tau_{k}} \pmod p$. The corollary follows from \cite{LubotzkySamuelsVishneIsospectralCayleyGraphs}, Theorem \ref{Thm:TwistCayleyExpMulti}, Proposition \ref{Prop:IsospecExample}.
\end{proof}

It would be worth investigating how many graphs among the graphs as in the proof of Corollary \ref{Cor:AsymptoticallyIsoSpecPerm} are non-isomorphic. In this direction, the following question can be asked. 

\begin{question}
Let $G$ denote the group $\frakS_n$ or $\frakA_n$. Let $\Sigma$ be a collection of order two inner automorphisms of $G$ and $S$ be a subset of $G$ such that $\sigma(S) = S^\mo$ for any $\sigma\in \Sigma$ (for instance, we could take $S$ to be the set of transpositions or $3$-cycles according as $G = \frakS_n$ or $G = \frakA_n$). 
\begin{enumerate}
\item 
How many of the twisted Cayley graphs $C(G, S)^\sigma$ are pairwise non-isomorphic as $\sigma$ varies over $\Sigma$? 
\item 
How many of the twisted Cayley graphs $C(G, S)^\sigma$ are pairwise non-isospectral as $\sigma$ varies over $\Sigma$?
\end{enumerate}
\end{question}

One can ask analogous questions for twisted Cayley sum graphs, and other groups (for instance, the special linear groups). A partial answer to some of these questions is provided by Proposition \ref{Prop:IsomTwistedCayley}.

\subsection{Uniformly $H$-isospectral $s$-non-isomorphic expander families}

We now proceed towards establishing that for any $2$-torsion group $H$, there are uniformly $H$-isospectral $s$-non-isomorphic expander families with $s = \log |H|$. 

As pointed out in Section \ref{Sec:TwistedCayley}, to obtain non-isomorphic expander families, we could consider the twisted Cayley graphs of a group $G$ with respect to inner automorphisms corresponding to two (or more) elements of $G$ lying in distinct conjugacy classes. 

\begin{lemma}
\label{Lemma:Pivotsubsets}
Let $A$ be a subset of $\{1, 2, \ldots, n\}$, and let $1\leq s \leq n$ be an integer. Consider the number of subsets $Y$ of $\{1, 2, \ldots, n\}$ satisfying the following conditions. 
\begin{enumerate}
\item The size of $Y$ is equal to $s$. 
\item The sets $Y, A$ are disjoint. 
\item There exists an element $\tau_Y \in \frakS_n$ such that for each $y\in Y$, 
\begin{enumerate}
\item $\tau_Y(y)$ lies in $A$, or 
\item $\tau_Y(y)$ lies in $Y$ and $\tau_Y(y) \neq y$. 
\end{enumerate}
\end{enumerate}
The number of such subsets is at most $2^{s-1} |A|^s (n-|A|)^{s/2}$.
\end{lemma}

\begin{proof}
Consider a subset $Y$ of $\{1, \ldots, n\}$ satisfying the given conditions. A $\ell$-tuple $(y_1, \ldots, y_\ell)$ consisting of distinct elements of $Y$ is said to be a \textit{chain in $Y$} if $\tau_Y(y_i) = y_{i+1}$ for $1\leq i < \ell$. Note that a $1$-tuple is a chain. Also note that a chain $(y_1, \ldots, y_\ell)$ in $Y$ is determined by its last coordinate, i.e., by $y_\ell$. A chain of maximal length in $Y$ is said to be a \textit{maximal chain in $Y$}. Note that a maximal chain in $Y$ could have any length between $1$ and $y$. Two chains in $Y$ are said to be \textit{disjoint} if their associated sets are disjoint. Note that two distinct maximal chains in $Y$ are disjoint, and the union of the distinct maximal chains of $Y$ is equal to $Y$. 

Let $t, r$ be integers satisfying $1 \leq t \leq s, 0 \leq r \leq t$. Consider the subsets $Y$ of $\{1, \ldots, n\}$ satisfying the given conditions and having exactly $t$ distinct maximal chains, among which exactly $r$ maximal chains are of length $1$ (i.e., a $1$-tuple). Note that $\tau_Y$ applied to such a $1$-tuple yields an element of $A$. Thus, to form such a subset, we need to determine the last coordinate of each of each of the $t$ maximal chains, which can be done by choosing $r$ distinct elements from $A$, and $t-r$ distinct elements of $\{1, 2, \ldots, n\}\setminus A$. So, the number of the subsets $Y$ of $\{1, \ldots, n\}$ satisfying the given conditions and having exactly $t$ distinct maximal chains (necessarily disjoint), among which exactly $r$ maximal chains are of length $1$ is 
$$\leq 
\sum_{x_1 + x_2 + \cdots + x_t = s , x_i \geq 1, r = |\{i \,|\, x_i = 1\}|}
\binom{|A|} {r} r!
\binom{n - |A| }{t-r} (t-r)!.$$
Hence, the number of the subsets $Y$ of $\{1, \ldots, n\}$ satisfying the given conditions and having exactly $t$ distinct maximal chains is 
\begin{align*}
& \leq 
\sum_{r= 0}^t
\sum_{x_1 + x_2 + \cdots + x_t = s , x_i \geq 1, r = |\{i \,|\, x_i = 1\}|}
\binom{|A|} {r} r!
\binom{n - |A| }{t-r} (t-r)!\\
& \leq 
\sum_{r= 0}^t
\sum_{x_1 + x_2 + \cdots + x_t = s , x_i \geq 1, r = |\{i \,|\, x_i = 1\}|}
|A|^r
(n - |A| )^{t-r}\\
& \leq 
\sum_{r= 0}^t
\sum_{x_1 + x_2 + \cdots + x_t = s , x_i \geq 1, r = |\{i \,|\, x_i = 1\}|}
|A|^r
(n - |A| )^{\sum_{i :  x_i > 1} 1}\\
& \leq 
\sum_{r= 0}^t
\sum_{x_1 + x_2 + \cdots + x_t = s , x_i \geq 1, r = |\{i \,|\, x_i = 1\}|}
|A|^r
(n - |A| )^{\sum_{i :  x_i > 1} x_i/2}\\
& \leq 
\sum_{r= 0}^t
\sum_{x_1 + x_2 + \cdots + x_t = s , x_i \geq 1, r = |\{i \,|\, x_i = 1\}|}
|A|^s
(n - |A| )^{s/2}\\
& \leq 
\sum_{x_1 + x_2 + \cdots + x_t = s , x_i \geq 1}
|A|^s
(n - |A| )^{s/2}\\
& \leq 
\binom{s-1}{t-1} 
|A|^s
(n - |A| )^{s/2}.
\end{align*}
This shows that the number of subsets $Y$ of $\{1, \ldots, n\}$ satisfying the given conditions is at most $2^{s-1} |A|^s (n-|A|)^{s/2}$.
\end{proof}

\begin{theorem}
\label{Thm:CountingVerticesAdmitting3Cycles}
For $n\geq 2$, let $G_n$ denote one of $\frakA_n, \frakS_n$. For each $n\geq 2$, let $F_n$ be a symmetric subset of $\frakA_n$ or $\frakS_n$ according as $G_n = \frakA_n$ or $G_n = \frakS_n$. Assume that the sets $F_n$ are of bounded size. Let $\pi_n$ denote an element of $\frakS_n$ that can be expressed as a product of $r$ pairwise disjoint transpositions $(i_{1n} ,j_{1n}), (i_{2n}, j_{2n}), \ldots, (i_{r n}, j_{r n})$ in $\frakS_n$ for some integer $r\geq 1$ (not depending on $n$). Let $\sigma_n$ denote the inner automorphism of $G_n$ induced by $\pi_n$. Assume that the twisted Cayley graph $C(G_n, F_n)^{\sigma_n}$ is undirected and $F_n^3$ contains 
$$(j_{1n}, j_{2n}, \ldots, j_{r n}, i_{r n}, \ldots, i_{2 n}, i_{1n}, \kappa_n)$$
for some integer $\kappa_n$ between $1$ and $n$ other than $i_{1n}, j_{1n},i_{2n}, j_{2n}, \ldots, i_{r n}, j_{ r n}$. Then the number of vertices of $C(G_n, F_n)^{\sigma_n}$ that admit a loop is at least $\frac 14 (n-2r)!$ and at most $K (n-r)!$ for some constant $K$ depending only on $r$ and the bound on the sizes of the sets $F_n$. 
\end{theorem}

\begin{proof}
Let $\kappa_n$ be an integer between $1$ and $n$ other than $i_{1n}, j_{1n},i_{2n}, j_{2n}, \ldots, i_{r n}, j_{ r n}$. Note that for the element 
$$g_n= (i_{1n}, j_{1n},i_{2n}, j_{2n}, \ldots, i_{r n}, j_{ r n} , \kappa_n)$$
in $G_n$ and for any $g'\in G_n$, which fixes the integers $i_{1n}, j_{1n},i_{2n}, j_{2n}, \ldots, i_{r n}, j_{ r n} , \kappa_n$, we have 
$$\sigma_n((g_n g')^\mo) g_n g'
= 
(j_{1n}, j_{2n}, \ldots, j_{r n}, i_{r n}, \ldots, i_{2 n}, i_{1n}, \kappa_n).$$
Moreover, for any $r$ distinct integers $a_1, \ldots, a_r$ between $1$ and $n$ other than $i_{1n}, j_{1n},i_{2n}, j_{2n}, \ldots, i_{r n}, j_{ r n} , \kappa_n$ and an element $g'\in G_n$ which fixes the integers $a_1, \ldots, a_r, i_{1n}, j_{1n},i_{2n}, j_{2n}, \ldots, i_{r n}, j_{ r n} , \kappa_n$, we have 
$$\sigma_n((h_n g')^\mo) h_n g'
= 
(j_{1n}, j_{2n}, \ldots, j_{r n}, i_{r n}, \ldots, i_{2 n}, i_{1n}, \kappa_n)$$
where 
$$h_n
= 
(i_{1n}, j_{1n},i_{2n}, j_{2n}, \ldots, i_{r n}, j_{ r n} , a_1, \ldots, a_r, \kappa_n)
.$$
It follows that there are at least $\frac 14 (n - 2r-1)!(n-2r) = \frac 14 (n-2r)!$ elements in $G_n$ which admit a loop in the graph $C(G_n, F_n)^{\sigma_n}$. 

Let $A$ denote the subset $\{i_{1n} ,j_{1n}, i_{2n}, j_{2n}, \ldots, i_{r n}, j_{r n}\}$ of $\{1, \ldots, n\}$. For any $x\in G_n$, let $\calY_x$ denote the subset defined by 
$$\calY_x: = 
\left(x^\mo (A)  \right)\cap (\{1, \ldots, n\} \setminus A).$$
Note that $\calY_x$ is a subset of $\{1, \ldots, n\}$ of size $\leq 2r$, and it is disjoint from $A$. Also note that for any $y\in \calY_x$, $(\sigma_n(x^\mo) x)^\mo (y)$ is either an element of $A$ or it is an element of $\calY_x$ other than $y$. Indeed, for an element $y\in \calY_x$, we have $y = x^\mo (a)$ for some $a\in A$, and 
$$(\sigma_n(x^\mo) x)^\mo (y)
= x^\mo (\sigma_n(x) (y))
= x^\mo (\sigma_n(a)) .$$
Note that $\sigma_n(a)$ lies in $A$, and hence it lies either in $A \setminus x(\calY_x)$ or in $x(\calY_x)$. If $\sigma_n(a)$ lies in $A \setminus x(\calY_x)$, then $x^\mo (\sigma_n(a))$ lies in 
\begin{align*}
x^\mo (A) \setminus \calY_x
& = 
\left(
(x^\mo (A) \cap A)
\cup 
(x^\mo (A) \cap (\{1, 2, \ldots, n\} \setminus A))
)
\right)
\setminus \calY_x\\
& = x^\mo (A) \cap A\\
& \subseteq A.
\end{align*}
If $\sigma_n(a)$ lies in $x(\calY_x)$, then $x^\mo (\sigma_n(a))$ lies in $\calY_x$, and further, $x^\mo (\sigma_n(a))\neq y$, otherwise, we would obtain $\sigma_n(a) = a$, which is impossible. 

Moreover, if $x\in G_n$ is an element admitting a loop in the graph $C(G_n, F_n)^{\sigma_n}$, then $\sigma(x^\mo) x$ is an element of $F_n^3$. So, the number of vertices of $C(G_n, F_n)^{\sigma_n}$ admitting loops is
\begin{align*}
& \leq \sum_{f_0\in F_n^3}\sum_{s = 0}^{2r} \sum_{\substack{Y \subseteq \{1, 2, \ldots, n\},\\ Y \cap A = \emptyset, |Y| = s}} \sum_{x\in G_n} |\{x\in G_n \,|\, \sigma(x^\mo) x = f_0, \calY_x = Y\}|\\
& \leq \sum_{f_0\in F_n^3}\sum_{s = 0}^{2r} \sum_{\substack{Y \subseteq \{1, 2, \ldots, n\},\\ Y \text{ satisfies the }\\
\text{ conditions of Lemma \ref{Lemma:Pivotsubsets}}
}} \sum_{x\in G_n} |\{x\in G_n \,|\, \sigma(x^\mo) x = f_0, \calY_x = Y\}|.
\end{align*}
Fix an element $f_0$ of $F_n^3$ and also fix a subset $\calY$ of $\{1, \ldots, n\}$ of size $s$ (for some $0\leq s \leq 2r$), which is disjoint from $A$. Note that the number of elements $x\in G_n$ such that $\sigma(x^\mo) x = f_0$ and $\calY_x$ is equal to $\calY$ is 
\begin{align*}
& \leq 
\binom{2r}{s} 
\binom{n-2r}{s} 
(2r)! 
s!(n-2r-s)! \frac 1{2^s}\\
& \leq 
\binom{2r}{s} 
(n-2r)!
(2r)! 
\frac 1{2^s}.
\end{align*}
Using Lemma \ref{Lemma:Pivotsubsets}, it follows that the number of vertices of $C(G_n, F_n)^{\sigma_n}$ admitting loops is
\begin{align*}
& \leq \sum_{f_0\in F_n^3}\sum_{s = 0}^{2r} \sum_{\substack{Y \subseteq \{1, 2, \ldots, n\},\\ Y \text{ satisfies the }\\
\text{ conditions of Lemma \ref{Lemma:Pivotsubsets}}
}} \sum_{x\in G_n} |\{x\in G_n \,|\, \sigma(x^\mo) x = f_0, \calY_x = Y\}|\\
& \leq \sum_{f_0\in F_n^3}
\sum_{s = 0}^{2r} 
\binom{2r}{s} 
2^{s-1}
(2r)^s
(n - 2r)^{s/2}
(n-2r)!
(2r)! 
\frac 1{2^s}\\
& \leq 
\frac {|F_n^3|}2 
\sum_{s = 0}^{2r} 
\binom{2r}{s} 
(n-2r)!
(2r)! 
(2r)^s
(n - 2r)^{s/2}\\
& \leq 
\frac {|F_n^3|}2 
(n-2r)!
(2r)! 
\sum_{s = 0}^{2r} 
\binom{2r}{s} 
(2r)^s
(n - 2r)^{s/2}\\
& \leq 
\frac {|F_n^3|}2 
(n-2r)!
(2r)! 
( 1+ 2r\sqrt{n-2r}) ^{2r} \\
& \leq 
K(n-r)!
\end{align*}
for some constant $K$ depending only on $r$ and the bound on the sizes of the sets $F_n$. 
\end{proof}

\begin{theorem}
\label{Thm:AsymptoticallyIsoSpecNonIsom}
For any $2$-torsion group $H$, there exists uniformly $H$-isospectral $\log|H|$-non-isomorphic expander families of bounded degree formed by the twisted Cayley graphs of symmetric groups, and by the twisted Cayley graphs of alternating groups. 
\end{theorem}

\begin{proof}
By \cite[Theorem 2]{KassabovSymmetricGrpExpander}, it follows that there exist constants $M>0$ and $\varepsilon>0$ and there exist subsets $F_n$ of the alternating group $\frakA_n$ and $\tilde F_n$ of the symmetric group $\frakS_n$ for any $n$ such that each of the Cayley graphs 
$C(\frakA_n, F_n)$, 
$C(\frakA_n, F_n), \ldots, 
C(\frakA_n, F_n)$, 
$C(\frakS_n, \tilde F_n)$, 
$C(\frakS_n, \tilde F_n), \ldots, 
C(\frakS_n, \tilde F_n)$, 
forms a family of $\varepsilon$-expanders of degree bounded by $M$. Consider the integers $n\geq 2^{2k}$. For $1\leq i \leq k$, let $\pi_{in}$ denote the element $(1, 2) (3, 4) \cdots (2^{2i-1}-1, 2^{2i-1})$ of $\frakS_n$. Note that the elements $\pi_{1n}, \ldots, \pi_{kn}$ of $\frakS_n$ are pairwise commuting, and $\pi_{in}$ is a product $2^{2i-2}$ pairwise disjoint transpositions in $\frakS_n$ for any $1\leq i \leq k$. Let $\sigma_{in}$ denote the inner automorphism of $\frakS_n$ induced by $\pi_{in}$. Let $S_n$ (resp. $\tilde S_n$) denote the subset of $\frakA_n$ (resp. $\frakS_n$) consisting of elements of $F_n$ (resp. $\tilde F_n$) and the elements 
$$(2, 4, 6, \ldots, 2^{2i-1}, 2^{2i-1}-1, \ldots, 5, 3, 1, 2^{2k-1} + 1)$$ 
for any $1\leq i \leq k$, and the inverses of these $k$ elements. Note that the Cayley graphs 
$C(\frakA_n, S_n)$, 
$C(\frakA_n, S_n), \ldots, 
C(\frakA_n, S_n)$, 
$C(\frakS_n, \tilde S_n)$, 
$C(\frakS_n, \tilde S_n), \ldots, 
C(\frakS_n, \tilde S_n)$, 
forms a family of $\varepsilon$-expanders of degree bounded by $M  +2k$. By Theorem \ref{Thm:TwistCayleyExpMulti}, there exist subsets $T_n$ of the alternating group $\frakA_n$ and $\tilde T_n$ of the symmetric group $\frakS_n$ for any $n\geq 2^{2k}$ such that the twisted Cayley graphs 
$C(\frakA_n, T_n)^{\sigma_{1n}}$, 
$C(\frakA_n, T_n)^{\sigma_{2n}}, \ldots, 
C(\frakA_n, T_n)^{\sigma_{kn}}$, 
$C(\frakS_n, \tilde T_n)^{\tilde \sigma_{1n}}$, 
$C(\frakS_n, \tilde T_n)^{\tilde \sigma_{2n}}, \ldots, 
C(\frakS_n, \tilde T_n)^{\tilde \sigma_{kn}}$, 
forms a family of $\varepsilon$-expanders of degree bounded by $2^k (M+2k)$. By Theorem \ref{Thm:CountingVerticesAdmitting3Cycles}, for each $1\leq i \leq k$, the number of vertices of $C(\frakA_n, T_n)^{\sigma_{in}}$ admitting loops lies between $\frac 14 (n-2^{2i-1})!$ and $K(n-2^{2i-2})!$. Hence for large enough $n$, the graphs $C(\frakA_n, T_n)^{\sigma_{in}}, 1\leq i \leq k$ are pairwise non-isomorphic. Similarly, by Theorem \ref{Thm:CountingVerticesAdmitting3Cycles}, it follows that for large enough $n$, the graphs $C(\frakS_n, \tilde T_n)^{\tilde \sigma_{in}}, 1\leq i \leq k$ are pairwise non-isomorphic. Applying Lemma \ref{Lemma:UniformOnSn}, the result follows. 
\end{proof}

The above discussion is also related to the following question. 

\begin{question}
For any integer $k\geq 2$, do there exist $k$ families of isospectral non-isomorphic expanders?
\end{question}

Further, one could ask if it is true that the (undirected) twisted Cayley graphs of symmetric groups and alternating groups with respect to twists by inner automorphisms induced by  order-two permutations ``coming from a fixed lower degree symmetric group'' are eventually non-isomorphic if the order-two permutations have distinct cycle types, and the twisted Cayley graphs are of bounded degree. 

\begin{question}
Let $G_n = \frakS_n$ or $G_n = \frakA_n$, and $S_n$ denote a symmetric subset of $G_n$ of bounded size. Let $k\geq 1$ be an integer and $\pi_1, \ldots, \pi_k$ be pairwise commuting elements of $\frakS_N$ of order two having distinct cycle types (where $N$ is some fixed integer). We will consider them as elements of $\frakS_n$ for any $n\geq N$ via the obvious inclusion map $\frakS_N  \to \frakS_n$. Let $\sigma_{1n}, \ldots, \sigma_{kn}$ be inner automorphisms of $G_n$ corresponding to $\pi_1, \ldots, \pi_k$. 
Let 
$$T_n
= S_n 
\cup 
\left(\cup_{i} \sigma_{in}(S_n) \right)
\cup 
\left(\cup_{i, j} (\sigma_{in} \circ \sigma_{jn}) (S_n) \right)
\cup 
\cdots 
\cup 
(\sigma_{1n} \circ \sigma_{2n} \circ \cdots \circ \sigma_{kn})(S_n).$$
Is it true that the twisted Cayley graphs $C(G_n, T_n)^{\sigma_{in}}, 1\leq i \leq k$, are non-isomorphic for any large enough $n$ if $S_n$ is of bounded size? 
\end{question}

Note that if the above-mentioned twisted Cayley graphs are considered as $G_n$-vertex graphs, then from Lemma \ref{Lemma:TwoSubgroupsAreEqual}, it follows that they are non-isomorphic for any large enough $n$ if $S_n$ is of bounded size.

\section{An application to diameters} 
\label{Sec:Diam}
The previous discussions have the following interesting consequences. 

\subsection{Diameter of abelian groups} The structural properties of groups, like diameters, play an important role in studying them. By the diameter, we shall mean the supremum over all shortest length paths between any two points in the graph. In the following we shall study the diameters of the Cayley sum graphs and the twisted graphs of abelian groups. In 1992, Babai--Seress showed that if $G$ is abelian, then the diameter of the Cayley graph of $G$ with respect to a generating set $S$ is bounded below by $\frac{|S|}{2e}|G|^{\frac{1}{|S|}}-|S|$ \cite{BabaiSeressDiamPermGrp}. A similar result holds for the Cayley sum graphs, the twisted Cayley graphs and the twisted Cayley sum graphs.   
\begin{theorem}
\label{Thm:DiamAbelian}
Let $\sigma$ be an order two automorphism of an abelian group $G$ and $|G|\geq 4$. Let $S$ be a symmetric subset of $G$ of size $d$. Let $X$ be a graph among the the Cayley sum graph $C_\Sigma(G, S)$, the twisted Cayley graph $C(G, S)^\sigma$ and the twisted Cayley sum graph $C_\Sigma(G, S)^\sigma$. Suppose $X$ is undirected. Then the diameter of $X$ is bounded below by 
$$\frac{|S|}{4e}|G|^{\frac{1}{|S|}}-\frac{1}{2}|S|.$$
\end{theorem}

\begin{proof}
By combining \cite[Proposition 1.1]{BabaiSeressDiamPermGrp} and Theorem \ref{Thm:EquivExpansion}, the result follows. 
\end{proof}

\subsection{Diameter of non-abelian finite simple groups}
The study of diameters of finite simple groups has been a thriving area of research. A fundamental question of Babai deals with bounding the diameter of the Cayley graphs of finite simple groups with respect to the size of the group. 

\begin{conjecture}[Babai, {\cite[Conjecture 1.7]{BabaiSeressDiamCayleySymm}}]\label{BabaiConj}
For any non-abelian finite simple group $G$, and for any symmetric generating set, the diameter of the Cayley graph is bounded above by $(\log |G|)^{O(1)}$, where the implied constant is absolute.
\end{conjecture}

By the classification theorem for finite simple group, there are two families to consider since we are concerned with asymptotic bounds, viz., the permutation groups and the finite simple groups of Lie type. In the case of the permutation groups $\frakS_n$ and the alternating group $\frakA_n$, it was shown by Helfgott--Seress  \cite[Main Theorem]{HelfgottSeressDiamPermGrp} that the diameters are bounded above by $\exp(O(\log n)^{4}\log \log n)$. For the finite simple groups of Lie type, Babai's conjecture has been shown to be true in the bounded rank case independently by the works of Breuillard--Green--Tao \cite{BreuillardGreenTaoApproxSubgrpLinGrp} and of Pyber--Szab\'{o} \cite{PyberSzaboGrowthFinSimpleGrpLie}. Further, when $G$ is a finite simple group of Lie type of rank $n$ defined over the field of size $q$, the diameter of the Cayley graph of $G$ with respect to any symmetric generating set $S$ is bounded above by $q^{O(n(\log n)^{3})}$. This follows from a previous work of the first author together with Yang \cite{BiswasYangDiamBddFSGLargeRk}. This has recently been improved to $q^{O(n(\log n)^{2})}$ by Halasi--Mar\'{o}ti--Pyber--Qiao \cite{HalasiMarotiPyberQiaoImprDiamBddFSGL}. From the above, we can conclude that

\begin{theorem}[Deterministic diameter bound]
\label{Thm:DiamBddDeterministic}
Let $S$ be a symmetric subset of a group $G$ of size $d$, and $\sigma$ be an order two automorphism of $G$. Let $X$ be a graph among the the Cayley sum graph $C_\Sigma(G, S)$, the twisted Cayley graph $C(G, S)^\sigma$ and the twisted Cayley sum graph $C_\Sigma(G, S)^\sigma$. Suppose $X$ is undirected. Then,
\begin{enumerate}
\item 
There exists an absolute constant $K>0$ such that for $G = \frakS_n$ and for $G = \frakA_n$, the diameter of $X$ is bounded above  by $ \exp(K\log^{4}n\log \log n)$. 

\item Given $r\geq 1$, there exists an absolute constant $K>0$ such that for any $q$, and $G$ denoting a finite simple group of Lie type of bounded rank $r$, over the field of size $q$, the diameter of $X$ is bounded above by $ (\log |G|)^{K}$.

\item 
There exists an absolute constant $K$ such that for any $n, q$ and $G$ denoting a finite simple group of Lie type of rank $n$ over the field of size $q$, the diameter of $X$  is bounded above by $q^{Kn(\log n)^{2}}$.
\end{enumerate}
\end{theorem}

\begin{proof}
By combining \cite[Main Theorem]{HelfgottSeressDiamPermGrp}, \cite{BreuillardGreenTaoApproxSubgrpLinGrp}, \cite{PyberSzaboGrowthFinSimpleGrpLie}, \cite{HalasiMarotiPyberQiaoImprDiamBddFSGL} with Theorem \ref{Thm:EquivExpansion}, the result follows. 
\end{proof}

From the classification of finite simple groups, it is known that a finite simple group is generated by a random pair of elements. So, it makes sense to study the diameter of the Cayley graph of these groups with respect to a random pair of elements generating the group. Babai's conjecture is known to hold true in this setting. See the works of Helfgott--Seress--Zuk \cite{HelfgottSeressZukRandomGenSymm}, Breuillard--Green--Tao \cite{BreuillardGreenTaoApproxSubgrpLinGrp}, Pyber--Szab\'{o} \cite{PyberSzaboGrowthFinSimpleGrpLie} and the very recent preprint of Eberhard--Jezernik \cite{EberhardJezernikBabaiConj}. 

\begin{theorem}[Diameter for the random generators of permutation groups]
\label{Thm:DiamBddRandomPerm}
Let $s,t$ denote a random pair of elements of a finite simple group $G$ and let $\sigma$ be an order two automorphism of $G$. 
There exist absolute constants $K, c>0$ such that for $G=\frakS_{n}$ and for $G=\frakA_{n}$, the diameter of $C(G, \{s^{\pm 1},t^{\pm 1}, \sigma(s^{\pm 1}), \sigma(t^{\pm 1})\})^\sigma$ is at most $Kn^{2}(\log n)^{c}$ with probability tending to $1$ as $n\rightarrow \infty$.
\end{theorem}

\begin{proof}
By combining \cite{HelfgottSeressZukRandomGenSymm} with Theorem \ref{Thm:EquivExpansion}, the result follows.
\end{proof}

\begin{theorem}
[Diameter for the random generators of classical groups]
\label{Thm:DiamBddRandomClassical}
\quad
\begin{enumerate}
\item Given $r\geq 1$, there exists an absolute constant $K>0$ such that for any $q$, and $G$ denoting a finite simple group of Lie type of bounded rank $r$, over the field of size $q$, the diameter of $C(G, \{x^{\pm 1}, y^{\pm 1}, \sigma(x^{\pm 1}), \sigma(y^{\pm 1})\})^\sigma$ is bounded above by $K\log |G|$  with probability $1 - o_{|G| \to \infty} (1)$ for any $x, y \in G$, chosen uniformly at random. 
	
\item 
There exists a positive constant $c$ such that for any $n\geq 2$, and any prime $p$ satisfying $\log p < cn /\log^2 n$, and for three elements $x, y, z$ of $\sln_n(\bbFp)$, chosen uniformly at random, the diameter of the twisted Cayley graph 
$$
C
(\sln_n(\bbFp), \{x^{\pm 1}, y^{\pm 1}, z^{\pm 1}, 
\sigma(x^{\pm 1}), \sigma(y^{\pm 1}), \sigma(z^{\pm 1})\})
^\sigma
$$
is at most $n^{O(\log p)}$ 
with probability $1- e^{-cn}$, for any order two automorphism $\sigma$ of $\sln_n(\bbFp)$. 

\item
There are positive constants $c, C$ such that for any $n> C$ and for any $k$ elements $x_1, \ldots, x_k$ of $\sln_n(\bbF_q)$, chosen uniformly at random with $k > q^C$, the diameter of the twisted Cayley graph 
$$
C(\sln_n(\bbF_q), 
\{
x_1^{\pm 1}, \ldots, x_k^{\pm 1}, 
\sigma(x_1^{\pm 1}), \ldots, \sigma(x_k^{\pm 1})\}
)^\sigma
$$
is at most $q^2 n^C$ with probability $1- q^{-cn}$, for any order two automorphism $\sigma$ of $\sln_n(\bbF_q)$. 
\end{enumerate}
\end{theorem}

\begin{proof}
By combining \cite[Theorem 7.2]{BreuillardGreenTaoApproxSubgrpLinGrp}, \cite[Theorems 1.2, 1.4]{EberhardJezernikBabaiConj} with Theorem \ref{Thm:EquivExpansion}, the result follows.
\end{proof}

We remark that there are results on diameters due to Gamburd--Shahshahani
\cite{GamburdShahshahaniUniformDiamBoundsFamilyCayley}, Kassabov--Riley \cite{KassabovRileyDiameterCayleyGraphChevalley}, Dinai \cite{DinaiDiameterChevalleyLocalRing}, Bradford \cite{BradfordNewUniformDiamBounds}. The above strategy can be used to yield bounds on the diameters of twisted Cayley graphs as an application of their results.

\section{Concluding remarks}
In the course of establishing the results about twisted Cayley graphs, we focused on those such that the connection set $S$ is closed under inversion and under the order two automorphism $\sigma$ of the underlying group $G$, i.e., the twisted Cayley graph $C(G, S)^\sigma$ and its corresponding Cayley graph $C(G, S)$ are both undirected. In certain cases, the connection set was chosen suitably so that these conditions hold. This was done in order to apply certain results about Cayley graphs in conjunction with Theorem \ref{Thm:EquivExpansion}, which allows us to compare a twisted Cayley graph and its corresponding Cayley graph under certain technical conditions, one of them being that both the graphs are undirected. 

It would be interesting to investigate whether the results obtained about the twisted Cayley graphs hold without the supplementary hypothesis that the associated Cayley graphs are undirected, i.e., the connection set is symmetric. A strong motivation to investigate this question comes from the fact that the twisted Cayley graphs and the Cayley graphs have certain common features (although there are significant differences). For instance, given a group $G$ and a subset $S$, to make the Cayley graph $C(G, S)$ symmetric, it suffices to replace $S$ by $S \cup S^\mo$, and thereby increasing the size of $S$ by at most twice of that of $S$, and similarly, to make the twisted Cayley graph $C(G, S)^\sigma$, it suffices to replace $S$ by $S\cup \sigma(S^\mo)$, and thereby increasing the size of $S$ by at most twice of that of $S$. Further, if $\sigma$ is an order two automorphism of $G$ and $S$ is a symmetric subset of $G$ with $\sigma(S) = S$, then the number of $2k$-cycles of $C(G, S)$ at any vertex is equal to the number of $2k$-cycles of $C(G, S)^\sigma$ at any vertex for any $k\geq1$. For further similarities, we refer to Theorems \ref{Thm:TwistCayleyAmenable}, \ref{Thm:TwistCayleySymmExp}, \ref{Thm:FSGLTwisted}. 

In addition, it would also be interesting to investigate whether the results obtained about the Cayley sum graphs and the twisted Cayley sum graphs hold without the supplementary hypothesis that the associated Cayley graphs are undirected. 

Moreover, while considering the twisted variants, in particular, the twisted Cayley graphs, it would be worth investigating the case when the group automorphism is of order $\geq 3$. 

\section{Acknowledgements}
The first author is supported by the ERC grant 716424 - CASe of K. Adiprasito. The second author would like to acknowledge the INSPIRE Faculty Award (IFA18-MA123) from the Department of Science and Technology, Government of India.

\newcommand{\etalchar}[1]{$^{#1}$}
\def\cprime{$'$} \def\Dbar{\leavevmode\lower.6ex\hbox to 0pt{\hskip-.23ex
  \accent"16\hss}D} \def\cfac#1{\ifmmode\setbox7\hbox{$\accent"5E#1$}\else
  \setbox7\hbox{\accent"5E#1}\penalty 10000\relax\fi\raise 1\ht7
  \hbox{\lower1.15ex\hbox to 1\wd7{\hss\accent"13\hss}}\penalty 10000
  \hskip-1\wd7\penalty 10000\box7}
  \def\cftil#1{\ifmmode\setbox7\hbox{$\accent"5E#1$}\else
  \setbox7\hbox{\accent"5E#1}\penalty 10000\relax\fi\raise 1\ht7
  \hbox{\lower1.15ex\hbox to 1\wd7{\hss\accent"7E\hss}}\penalty 10000
  \hskip-1\wd7\penalty 10000\box7}
  \def\polhk#1{\setbox0=\hbox{#1}{\ooalign{\hidewidth
  \lower1.5ex\hbox{`}\hidewidth\crcr\unhbox0}}}
\providecommand{\bysame}{\leavevmode\hbox to3em{\hrulefill}\thinspace}
\providecommand{\MR}{\relax\ifhmode\unskip\space\fi MR }
\providecommand{\MRhref}[2]{%
  \href{http://www.ams.org/mathscinet-getitem?mr=#1}{#2}
}
\providecommand{\href}[2]{#2}


\begin{thebibliography}{HMPQ19}

\bibitem[AB18]{ArzhantsevaBiswasDgBdd}
Goulnara Arzhantseva and Arindam Biswas, \emph{Large girth graphs with bounded
  diameter-by-girth ratio}, Preprint available at
  \url{https://arxiv.org/abs/1803.09229}, 2018.

\bibitem[Alo21]{AlonExplicitExpandersEveryDegSize}
Noga Alon, \emph{Explicit {E}xpanders of {E}very {D}egree and {S}ize},
  Combinatorica (2021).

\bibitem[AR94]{AlonRoichmanRandomCayley}
Noga Alon and Yuval Roichman, \emph{Random {C}ayley graphs and expanders},
  Random Structures Algorithms \textbf{5} (1994), no.~2, 271--284. \MR{1262979}

\bibitem[ASS08]{AlonSchwartzShapiraElemConsContDegreeExpander}
Noga Alon, Oded Schwartz, and Asaf Shapira, \emph{An elementary construction of
  constant-degree expanders}, Combin. Probab. Comput. \textbf{17} (2008),
  no.~3, 319--327. \MR{2410389}

\bibitem[AT16]{AmooshahiTaeri}
Marzieh Amooshahi and Bijan Taeri, \emph{On {C}ayley sum graphs of non-abelian
  groups}, Graphs Combin. \textbf{32} (2016), no.~1, 17--29. \MR{3436946}

\bibitem[Bab79]{BabaiSpectraofCayley}
L\'{a}szl\'{o} Babai, \emph{Spectra of {C}ayley graphs}, J. Combin. Theory Ser.
  B \textbf{27} (1979), no.~2, 180--189. \MR{546860}

\bibitem[BG08a]{BourgainGamburdExpansionRandomWalkSLdI}
Jean Bourgain and Alex Gamburd, \emph{Expansion and random walks in {${\rm
  SL}_d(\Bbb Z/p^n\Bbb Z)$}. {I}}, J. Eur. Math. Soc. (JEMS) \textbf{10}
  (2008), no.~4, 987--1011. \MR{2443926}

\bibitem[BG08b]{BourgainGamburdUniformExpanBdd}
\bysame, \emph{Uniform expansion bounds for {C}ayley graphs of {${\rm
  SL}_2(\Bbb F_p)$}}, Ann. of Math. (2) \textbf{167} (2008), no.~2, 625--642.
  \MR{2415383}

\bibitem[BG09]{BourgainGamburdExpansionRandomWalkSLdII}
\bysame, \emph{Expansion and random walks in {${\rm SL}_d(\Bbb Z/p^n\Bbb Z)$}.
  {II}}, J. Eur. Math. Soc. (JEMS) \textbf{11} (2009), no.~5, 1057--1103, With
  an appendix by Bourgain. \MR{2538500}

\bibitem[BG10]{BreuillardGamburdStrongUnifExpans}
Emmanuel Breuillard and Alex Gamburd, \emph{Strong uniform expansion in {${\rm
  SL}(2,p)$}}, Geom. Funct. Anal. \textbf{20} (2010), no.~5, 1201--1209.
  \MR{2746951}

\bibitem[BGGT15]{BGGTExpansionSimpleLie}
Emmanuel Breuillard, Ben Green, Robert Guralnick, and Terence Tao,
  \emph{Expansion in finite simple groups of {L}ie type}, J. Eur. Math. Soc.
  (JEMS) \textbf{17} (2015), no.~6, 1367--1434. \MR{3353804}

\bibitem[BGS10]{BourgainGamburdSarnakAffineLinearSieveExpanderSumProd}
Jean Bourgain, Alex Gamburd, and Peter Sarnak, \emph{Affine linear sieve,
  expanders, and sum-product}, Invent. Math. \textbf{179} (2010), no.~3,
  559--644. \MR{2587341}

\bibitem[BGT11a]{BreuillardGreenTaoApproxSubgrpLinGrp}
Emmanuel Breuillard, Ben Green, and Terence Tao, \emph{Approximate subgroups of
  linear groups}, Geom. Funct. Anal. \textbf{21} (2011), no.~4, 774--819.
  \MR{2827010}

\bibitem[BGT11b]{BGTSuzuki}
\bysame, \emph{Suzuki groups as expanders}, Groups Geom. Dyn. \textbf{5}
  (2011), no.~2, 281--299. \MR{2782174}

\bibitem[BHK{\etalchar{+}}90]{BabaiHetyeiKantorLubotzkySeressDiamFiniteGrp}
L.~Babai, G.~Hetyei, W.~M. Kantor, A.~Lubotzky, and \'{A}. Seress, \emph{On the
  diameter of finite groups}, 31st {A}nnual {S}ymposium on {F}oundations of
  {C}omputer {S}cience, {V}ol. {I}, {II} ({S}t. {L}ouis, {MO}, 1990), IEEE
  Comput. Soc. Press, Los Alamitos, CA, 1990, pp.~857--865. \MR{1150735}

\bibitem[Bis19]{BiswasCheegerCayley}
Arindam Biswas, \emph{On a {C}heeger type inequality in {C}ayley graphs of
  finite groups}, European J. Combin. \textbf{81} (2019), 298--308.
  \MR{3975766}

\bibitem[Bra16]{BradfordExpansionRandomWalkSieving}
Henry Bradford, \emph{Expansion, random walks and sieving in
  {$S{L_2}({\Bbb{F}_p}[t])$}}, Israel J. Math. \textbf{215} (2016), no.~2,
  559--582. \MR{3552289}

\bibitem[Bra18]{BradfordNewUniformDiamBounds}
\bysame, \emph{New uniform diameter bounds in pro-{$p$} groups}, Groups Geom.
  Dyn. \textbf{12} (2018), no.~3, 803--836. \MR{3845710}

\bibitem[BS]{CheegerCayleySum}
Arindam Biswas and Jyoti~Prakash Saha, \emph{A {C}heeger type inequality in
  finite {C}ayley sum graphs}, Algebr. Comb. (accepted),
  \url{https://doi.org/10.5802/alco.166}.

\bibitem[BS88]{BabaiSeressDiamCayleySymm}
L\'{a}szl\'{o} Babai and \'{A}kos Seress, \emph{On the diameter of {C}ayley
  graphs of the symmetric group}, J. Combin. Theory Ser. A \textbf{49} (1988),
  no.~1, 175--179. \MR{957215}

\bibitem[BS92]{BabaiSeressDiamPermGrp}
\bysame, \emph{On the diameter of permutation groups}, European J. Combin.
  \textbf{13} (1992), no.~4, 231--243. \MR{1179520}

\bibitem[BS20]{CheegerTwisted}
Arindam Biswas and Jyoti~Prakash Saha, \emph{Spectra of twists of {C}ayley
  graphs and {C}ayley sum graphs by automorphisms}, Preprint available at
  \url{https://arxiv.org/abs/2008.04307}, 2020.

\bibitem[BV12]{BourgainVarjuExpansionSLdZqZ}
Jean Bourgain and P\'{e}ter~P. Varj\'{u}, \emph{Expansion in {$SL_d({\bf
  Z}/q{\bf Z}),\,q$} arbitrary}, Invent. Math. \textbf{188} (2012), no.~1,
  151--173. \MR{2897695}

\bibitem[BY13]{BourgainYehudayoffExpansionInSL2MonotoneExpa}
Jean Bourgain and Amir Yehudayoff, \emph{Expansion in {${\rm SL}_2(\Bbb{R})$}
  and monotone expanders}, Geom. Funct. Anal. \textbf{23} (2013), no.~1, 1--41.
  \MR{3037896}

\bibitem[BY17]{BiswasYangDiamBddFSGLargeRk}
Arindam Biswas and Yilong Yang, \emph{A diameter bound for finite simple groups
  of large rank}, J. Lond. Math. Soc. (2) \textbf{95} (2017), no.~2, 455--474.
  \MR{3656277}

\bibitem[Chi92]{ChiuCubicRamanujan}
Patrick Chiu, \emph{Cubic {R}amanujan graphs}, Combinatorica \textbf{12}
  (1992), no.~3, 275--285. \MR{1195890}

\bibitem[Chu89]{Chung89JAMS}
F.~R.~K. Chung, \emph{Diameters and eigenvalues}, J. Amer. Math. Soc.
  \textbf{2} (1989), no.~2, 187--196. \MR{965008}

\bibitem[CM08]{ChristofidesMarkstromExpanRandomCayley}
Demetres Christofides and Klas Markstr\"{o}m, \emph{Expansion properties of
  random {C}ayley graphs and vertex transitive graphs via matrix martingales},
  Random Structures Algorithms \textbf{32} (2008), no.~1, 88--100. \MR{2371053}

\bibitem[Din12]{DinaiDiameterChevalleyLocalRing}
Oren Dinai, \emph{Diameters of {C}hevalley groups over local rings}, Arch.
  Math. (Basel) \textbf{99} (2012), no.~5, 417--424. \MR{3000421}

\bibitem[DS81]{DiaconisShahshahaniGeneRandPermTrans}
Persi Diaconis and Mehrdad Shahshahani, \emph{Generating a random permutation
  with random transpositions}, Z. Wahrsch. Verw. Gebiete \textbf{57} (1981),
  no.~2, 159--179. \MR{626813}

\bibitem[DS08]{DvirShpilkaTowardsDimExpanderFiniteField}
Zeev Dvir and Amir Shpilka, \emph{Towards dimension expanders over finite
  fields}, Twenty-{T}hird {A}nnual {IEEE} {C}onference on {C}omputational
  {C}omplexity, IEEE Computer Soc., Los Alamitos, CA, 2008, pp.~304--310.
  \MR{2500345}

\bibitem[DW10]{DvirWigdersonMonotoneExpanders}
Zeev Dvir and Avi Wigderson, \emph{Monotone expanders: constructions and
  applications}, Theory Comput. \textbf{6} (2010), 291--308. \MR{2770077}

\bibitem[EJ20]{EberhardJezernikBabaiConj}
Sean Eberhard and Urban Jezernik, \emph{Babai's conjecture for high-rank
  classical groups with random generators}, Preprint available at
  \url{https://arxiv.org/abs/2005.09990}, 2020.

\bibitem[GG81]{GabberGalilLinearSized}
Ofer Gabber and Zvi Galil, \emph{Explicit constructions of linear-sized
  superconcentrators}, J. Comput. System Sci. \textbf{22} (1981), no.~3,
  407--420, Special issued dedicated to Michael Machtey. \MR{633542}

\bibitem[GLS09]{GrynkiewiczLevSerraConnCaylSum}
David Grynkiewicz, Vsevolod~F. Lev, and Oriol Serra, \emph{Connectivity of
  addition {C}ayley graphs}, J. Combin. Theory Ser. B \textbf{99} (2009),
  no.~1, 202--217. \MR{2467826}

\bibitem[GS04]{GamburdShahshahaniUniformDiamBoundsFamilyCayley}
Alex Gamburd and Mehrdad Shahshahani, \emph{Uniform diameter bounds for some
  families of {C}ayley graphs}, Int. Math. Res. Not. (2004), no.~71,
  3813--3824. \MR{2104475}

\bibitem[GV12]{GolsefidyVarjuExpansionPerfectGrp}
A.~Salehi Golsefidy and P\'{e}ter~P. Varj\'{u}, \emph{Expansion in perfect
  groups}, Geom. Funct. Anal. \textbf{22} (2012), no.~6, 1832--1891.
  \MR{3000503}

\bibitem[Hel08]{HelfgottGrowthGenerationSL2Fp}
H.~A. Helfgott, \emph{Growth and generation in {${\rm SL}_2(\Bbb Z/p\Bbb Z)$}},
  Ann. of Math. (2) \textbf{167} (2008), no.~2, 601--623. \MR{2415382}

\bibitem[Hel15]{HelfgottGrowthInGroupsIdeaPersp}
Harald~A. Helfgott, \emph{Growth in groups: ideas and perspectives}, Bull.
  Amer. Math. Soc. (N.S.) \textbf{52} (2015), no.~3, 357--413. \MR{3348442}

\bibitem[Hel19]{HelfgottGrowthExpansionAlgGrpFinField}
Harald~Andr\'{e}s Helfgott, \emph{Growth and expansion in algebraic groups over
  finite fields}, Analytic methods in arithmetic geometry, Contemp. Math., vol.
  740, Amer. Math. Soc., Providence, RI, 2019, pp.~71--111. \MR{4033730}

\bibitem[HLW06]{HooryLinialWigdersonBAMS}
Shlomo Hoory, Nathan Linial, and Avi Wigderson, \emph{Expander graphs and their
  applications}, Bull. Amer. Math. Soc. (N.S.) \textbf{43} (2006), no.~4,
  439--561. \MR{2247919}

\bibitem[HMPQ19]{HalasiMarotiPyberQiaoImprDiamBddFSGL}
Zolt\'{a}n Halasi, Attila Mar\'{o}ti, L\'{a}szl\'{o} Pyber, and Youming Qiao,
  \emph{An improved diameter bound for finite simple groups of {L}ie type},
  Bull. Lond. Math. Soc. \textbf{51} (2019), no.~4, 645--657. \MR{3990383}

\bibitem[HS14]{HelfgottSeressDiamPermGrp}
Harald~A. Helfgott and \'{A}kos Seress, \emph{On the diameter of permutation
  groups}, Ann. of Math. (2) \textbf{179} (2014), no.~2, 611--658. \MR{3152942}

\bibitem[HSZ15]{HelfgottSeressZukRandomGenSymm}
Harald~A. Helfgott, \'{A}kos Seress, and Andrzej Zuk, \emph{Random generators
  of the symmetric group: diameter, mixing time and spectral gap}, J. Algebra
  \textbf{421} (2015), 349--368. \MR{3272386}

\bibitem[Kas07]{KassabovSymmetricGrpExpander}
Martin Kassabov, \emph{Symmetric groups and expander graphs}, Invent. Math.
  \textbf{170} (2007), no.~2, 327--354. \MR{2342639}

\bibitem[KB67]{KolmogorovBarzdin}
A.~N. Kolmogorov and Y.M. Barzdin, \emph{On the realization of nets in
  $3$-dimensional space}, Probl. Cybernet \textbf{8} (1967), 261--268, See also
  Selected Works of A. N. Kolmogorov, Vol. 3, pp. 194--202 (and a remark on
  page 245), Kluwer Academic Publishers, 1993.

\bibitem[KLN06]{KassabovLubotzkyNikolovFiniteSimpleGrpExp}
Martin Kassabov, Alexander Lubotzky, and Nikolay Nikolov, \emph{Finite simple
  groups as expanders}, Proc. Natl. Acad. Sci. USA \textbf{103} (2006), no.~16,
  6116--6119. \MR{2221038}

\bibitem[Kow13]{KowalskiExplicitGrowthExpansionSL2}
Emmanuel Kowalski, \emph{Explicit growth and expansion for {${\rm SL}_2$}},
  Int. Math. Res. Not. IMRN (2013), no.~24, 5645--5708. \MR{3144176}

\bibitem[KR07]{KassabovRileyDiameterCayleyGraphChevalley}
M.~Kassabov and T.~R. Riley, \emph{Diameters of {C}ayley graphs of {C}hevalley
  groups}, European J. Combin. \textbf{28} (2007), no.~3, 791--800.
  \MR{2300760}

\bibitem[Lev10]{LevSumDiffHamiltCycle}
Vsevolod~F. Lev, \emph{Sums and differences along {H}amiltonian cycles},
  Discrete Math. \textbf{310} (2010), no.~3, 575--584. \MR{2564813}

\bibitem[LPS88]{LPS88Ramanujan}
A.~Lubotzky, R.~Phillips, and P.~Sarnak, \emph{Ramanujan graphs}, Combinatorica
  \textbf{8} (1988), no.~3, 261--277. \MR{963118}

\bibitem[LR04]{LandauRussellRandomCayley}
Zeph Landau and Alexander Russell, \emph{Random {C}ayley graphs are expanders:
  a simple proof of the {A}lon-{R}oichman theorem}, Electron. J. Combin.
  \textbf{11} (2004), no.~1, Research Paper 62, 6. \MR{2097328}

\bibitem[LS01]{LiebeckShalevDiamFiniSimpleGrp}
Martin~W. Liebeck and Aner Shalev, \emph{Diameters of finite simple groups:
  sharp bounds and applications}, Ann. of Math. (2) \textbf{154} (2001), no.~2,
  383--406. \MR{1865975}

\bibitem[LS04]{LohSchulmanRandomCayley}
Po-Shen Loh and Leonard~J. Schulman, \emph{Improved expansion of random
  {C}ayley graphs}, Discrete Math. Theor. Comput. Sci. \textbf{6} (2004),
  no.~2, 523--528. \MR{2180056}

\bibitem[LSV06]{LubotzkySamuelsVishneIsospectralCayleyGraphs}
Alexander Lubotzky, Beth Samuels, and Uzi Vishne, \emph{Isospectral {C}ayley
  graphs of some finite simple groups}, Duke Math. J. \textbf{135} (2006),
  no.~2, 381--393. \MR{2267288}

\bibitem[Lub94]{LubotzkyDiscreteGroups}
Alexander Lubotzky, \emph{Discrete groups, expanding graphs and invariant
  measures}, Progress in Mathematics, vol. 125, Birkh\"{a}user Verlag, Basel,
  1994, With an appendix by Jonathan D. Rogawski. \MR{1308046}

\bibitem[Lub95]{LubotzkyCayleyGraphSurvey}
\bysame, \emph{Cayley graphs: eigenvalues, expanders and random walks}, Surveys
  in combinatorics, 1995 ({S}tirling), London Math. Soc. Lecture Note Ser.,
  vol. 218, Cambridge Univ. Press, Cambridge, 1995, pp.~155--189. \MR{1358635}

\bibitem[Lub10]{LubotzkyDiscreteGroups2010}
\bysame, \emph{Discrete groups, expanding graphs and invariant measures},
  Modern Birkh\"{a}user Classics, Birkh\"{a}user Verlag, Basel, 2010, With an
  appendix by Jonathan D. Rogawski, Reprint of the 1994 edition. \MR{2569682}

\bibitem[Lub12]{LubotzkyExpanderGraphsBAMS}
\bysame, \emph{Expander graphs in pure and applied mathematics}, Bull. Amer.
  Math. Soc. (N.S.) \textbf{49} (2012), no.~1, 113--162. \MR{2869010}

\bibitem[Lub18]{LubotzkyICMHighDimensionalExpander}
\bysame, \emph{High dimensional expanders}, Proceedings of the {I}nternational
  {C}ongress of {M}athematicians---{R}io de {J}aneiro 2018. {V}ol. {I}.
  {P}lenary lectures, World Sci. Publ., Hackensack, NJ, 2018, pp.~705--730.
  \MR{3966743}

\bibitem[LW93]{LubotzkyWeissGroupsAndExpanders}
A.~Lubotzky and B.~Weiss, \emph{Groups and expanders}, Expanding graphs
  ({P}rinceton, {NJ}, 1992), DIMACS Ser. Discrete Math. Theoret. Comput. Sci.,
  vol.~10, Amer. Math. Soc., Providence, RI, 1993, pp.~95--109. \MR{1235570}

\bibitem[LZ08]{LubotzkyZelmanovDimensionExpanders}
Alexander Lubotzky and Efim Zelmanov, \emph{Dimension expanders}, J. Algebra
  \textbf{319} (2008), no.~2, 730--738. \MR{2381805}

\bibitem[Mar73]{MargulisExpanders}
G.~A. Margulis, \emph{Explicit constructions of expanders}, Problemy
  Pereda\v{c}i Informacii \textbf{9} (1973), no.~4, 71--80. \MR{0484767}

\bibitem[Mar88]{Margulis88ExpandConcent}
\bysame, \emph{Explicit group-theoretic constructions of combinatorial schemes
  and their applications in the construction of expanders and concentrators},
  Problemy Peredachi Informatsii \textbf{24} (1988), no.~1, 51--60. \MR{939574}

\bibitem[Mor94]{MorgensternJCTB94}
Moshe Morgenstern, \emph{Existence and explicit constructions of {$q+1$}
  regular {R}amanujan graphs for every prime power {$q$}}, J. Combin. Theory
  Ser. B \textbf{62} (1994), no.~1, 44--62. \MR{1290630}

\bibitem[MRT20]{CayleyBottomBipartite}
Nina Moorman, Peter Ralli, and Prasad Tetali, \emph{On the {B}ipartiteness
  {C}onstant and {E}xpansion of {C}ayley {G}raphs}, Preprint available at
  \url{https://arxiv.org/abs/2008.05911}, 2020.

\bibitem[MSS15]{MarcusSpielmanSrivastavaInterlaFam1}
Adam~W. Marcus, Daniel~A. Spielman, and Nikhil Srivastava, \emph{Interlacing
  families {I}: {B}ipartite {R}amanujan graphs of all degrees}, Ann. of Math.
  (2) \textbf{182} (2015), no.~1, 307--325. \MR{3374962}

\bibitem[Pak99]{PakRandomCayleyGraphOlogG}
Igor Pak, \emph{Random {C}ayley graphs with {$O(\log|G|)$} generators are
  expanders}, Algorithms---{ESA} '99 ({P}rague), Lecture Notes in Comput. Sci.,
  vol. 1643, Springer, Berlin, 1999, pp.~521--526. \MR{1729149}

\bibitem[Pin73]{PinskerComplexityConcentrator}
Mark~S. Pinsker, \emph{On the complexity of a concentrator}, 7th International
  Teletraffic Conference, 1973.

\bibitem[Piz90]{PizerRamanujanGraphHeckeOperators}
Arnold~K. Pizer, \emph{Ramanujan graphs and {H}ecke operators}, Bull. Amer.
  Math. Soc. (N.S.) \textbf{23} (1990), no.~1, 127--137. \MR{1027904}

\bibitem[PS16]{PyberSzaboGrowthFinSimpleGrpLie}
L\'{a}szl\'{o} Pyber and Endre Szab\'{o}, \emph{Growth in finite simple groups
  of {L}ie type}, J. Amer. Math. Soc. \textbf{29} (2016), no.~1, 95--146.
  \MR{3402696}

\bibitem[Roi97]{RoichmanExpansionCayleyAlt}
Yuval Roichman, \emph{Expansion properties of {C}ayley graphs of the
  alternating groups}, J. Combin. Theory Ser. A \textbf{79} (1997), no.~2,
  281--297. \MR{1462559}

\bibitem[RSW06]{RozenmanShalevWigdersonIterative}
Eyal Rozenman, Aner Shalev, and Avi Wigderson, \emph{Iterative construction of
  {C}ayley expander graphs}, Theory Comput. \textbf{2} (2006), 91--120.
  \MR{2322872}

\bibitem[RVW02]{ReingoldVadhanWigdersonEntropyZigZag}
Omer Reingold, Salil Vadhan, and Avi Wigderson, \emph{Entropy waves, the
  zig-zag graph product, and new constant-degree expanders}, Ann. of Math. (2)
  \textbf{155} (2002), no.~1, 157--187. \MR{1888797}

\bibitem[Sha97]{ShalomExpandingGraphsInvariantMeans}
Yehuda Shalom, \emph{Expanding graphs and invariant means}, Combinatorica
  \textbf{17} (1997), no.~4, 555--575. \MR{1645694}

\bibitem[Sha99]{ShalomExpandingGraphsAmenableQuotients}
\bysame, \emph{Expander graphs and amenable quotients}, Emerging applications
  of number theory ({M}inneapolis, {MN}, 1996), IMA Vol. Math. Appl., vol. 109,
  Springer, New York, 1999, pp.~571--581. \MR{1691549}

\bibitem[Som15]{SomlaiNonExpanderCayley}
G\'{a}bor Somlai, \emph{Non-expander {C}ayley graphs of simple groups}, Comm.
  Algebra \textbf{43} (2015), no.~3, 1156--1175. \MR{3298126}

\bibitem[Var12]{VarjuExpansionSLdModISqFree}
P\'{e}ter~P. Varj\'{u}, \emph{Expansion in {$SL_d(\mathscr O_K/I)$}, {$I$}
  square-free}, J. Eur. Math. Soc. (JEMS) \textbf{14} (2012), no.~1, 273--305.
  \MR{2862040}

\end{thebibliography}
\end{document}